\documentclass[11pt, leqno]{article}
\usepackage{mathrsfs,amssymb,amsmath,amsthm,amsfonts}

\usepackage{fancyhdr,indentfirst,leftidx}

\usepackage{color}
\usepackage[colorlinks, linkcolor=red, anchorcolor=green, citecolor=blue]{hyperref}
\allowdisplaybreaks

\date{}
\allowdisplaybreaks

\newtheorem{theorem}{Theorem}[section]
\newtheorem{lemma}[theorem]{Lemma}

\newtheorem{remark}[theorem]{Remark}
\newtheorem{definition}[theorem]{Definition}
\numberwithin{equation}{section}
\begin{document}

\title{\Large\bf Multilinear Operators on Weighted Amalgam-Type Spaces
\footnotetext{\hspace{-0.35cm} {\it 2010 Mathematics Subject Classification}. {42B20,42B25, 42B35}
\endgraf{\it Key words and phrases.} Amalgam Spaces, $A_{\vec{p}}$ weight, Multilinear Operators, multilinear fractional type integral operator, $A_{\vec{p},q}$ weight, Amalgam-Campanato space, BMO.
\endgraf Songbai Wang is supported by Young Foundation of Education Department of Hubei Province (No.Q20162504).
}}
\author{Songbai Wang, Peng Li}
\maketitle
\begin{abstract}
In this paper, we prove that if a multilinear operator $\mathcal{T}$ and its multilinear commutator $\mathcal{T}_{\Sigma\vec{b}}$ and  iterated commutator $\mathcal{T}_{\Pi\vec{b}}$ for $\vec{b}\in(\mathbb{R}^n)^m$ are bounded on product weighted Lebesgue space, then $\mathcal{T}$, $\mathcal{T}_{\Sigma\vec{b}}$ and $\mathcal{T}_{\Pi\vec{b}}$ are also bounded on product weighted Amalgam space. As its applications, we show that multilinear Littlewood-Paley functions and multilinear Marcinkiewicz integral functions with kernels of convolution type and non-convolution type, and their multilinear commutators and iterated commutators are bounded on product weighted Amalgam space. We also consider multilinear fractional type integral operators and their commutators' behaviors  on weighted amalgam space. In order to deal with the endpoint case, we introduce the amalgam-Campanato space and show that fractional integral integral operator are bounded operators from product amalgam space to amalgam-Campanato space. What should point out is that even if in the linear case, our results for fractional integral operator are also new.
\end{abstract}


\section{Introduction}\label{s1}
\hskip\parindent

If $1\leq p, q\leq\infty,$ a function $f \in L^q_{loc}(\mathbb{R}^n)$ is said to be in the amalgam spaces $(L^q,L^p)(\mathbb{R}^n)$ of
$L^q(\mathbb{R}^n)$ and $L^p(\mathbb{R}^n)$ if $\|f(\cdot)\chi_{B(y,1)}(\cdot)\|_q$ belongs to $L^p(\mathbb{R}^n),$ where $B(y,r)$ denotes the open
ball with center $y$ and radius $r$ and $\chi_{B(y,r)}$ is the characteristic function of the ball $B(y,r)$,
here $\|\cdot\|_q$ is the usual Lebesgue norm in $L^q(\mathbb{R}^n).$
$$\|f\|_{q,p}=\bigg(\int_{\mathbb{R}^n}\|f(\cdot)\chi_{B(y,1)}(\cdot)\|_q^pdy\bigg)^{1/p}$$
is a norm on $(L^q,L^p)(\mathbb{R}^n)$ under which it is a Banach space with the usual modification
when $p=\infty.$ These amalgam spaces were first introduced by Winer \cite{W1926} in 1926 and its systematic study
goes back to the work of Holland \cite{H1975}.

In the study of the continuity of the fractional maximal operator of Hardy-Littlewood and of the Fourier transformation on $\mathbb{R}^n,$ Fofana generized the above amalgam space to $(L^q,L^p)^\alpha$, $1\leq q \leq \alpha \leq p\leq\infty$ in \cite{Fo1988,Fo1989}. The authors generalized the amalgam space to the spaces of homogeneous type \cite{FFK2003}and to the setting of homogeneous groups \cite{FFK2010} later on. Recently,  Feuto \cite{F2014} studied the weighted version of these amalgam spaces. In this paper, we make a little modification of Feuto's amalgam space.

\begin{definition}
Let $\omega_1,\omega_2$ be weights and $0< q\leq \alpha\leq p\leq\infty.$ We define the weighted Amalgam space $(L^q(\omega_1,\omega_2),L^p)^\alpha:=(L^q(\omega_1,\omega_2),L^p)^\alpha(\mathbb R^n)$ as the space of all measurable functions $f$ satisfying $\|f\|_{(L^q(\omega_1,\omega_2),L^p)^\alpha}<\infty$, where
$$
\|f\|_{(L^q(\omega_1,\omega_2),L^p)^\alpha}:=\underset{r>0}{\sup}\|f\|_{(L^q(\omega_1,\omega_2),L^p)^\alpha}
$$
with
\begin{align*}
\|f\|_{(L^q(\omega_1,\omega_2),L^p)^{\alpha}}:=
\begin{cases}
\bigg(\int_{\mathbb{R}^n}
\big(\omega_1(B(y,r))^{\frac{1}{\alpha}-\frac{1}{q}-\frac{1}{p}}\|f\chi_{B(y,r)}\|_{L^q(\omega_2)}\big)^pdy\bigg)^{\frac{1}{p}},\quad &p<\infty,\\
\mathop{\rm{ess}\sup}\limits_{y\in\mathbb{R}^n}\omega_1(B(y,r))^{\frac{1}{\alpha}-\frac{1}{q}}\|f\chi_{B(y,r)}\|_{L^q(\omega_2)}, \quad &p=\infty.
\end{cases}
\end{align*}
We also put out that for $r>0$
\begin{align*}
\|f\|_{(L^{q,\infty}(\omega_1,\omega_2),L^p)^{\alpha}}:=
\begin{cases}
\bigg(\int_{\mathbb R^n}\big(\omega_1(B(y,r))^{\frac1\alpha-\frac1q-\frac1p}\|f\chi_{B(y,r)}\|_{L^{q,\infty}(\omega_2)}\big)^pdy\bigg)^{\frac{1}{p}},
\quad &p<\infty,\\
\mathop{\rm{ess}\sup}\limits_{y\in\mathbb{R}^n}\omega_1(B(y,r))^{\frac{1}{\alpha}-\frac{1}{q}}\|f\chi_{B(y,r)}\|_{L^{q,\infty}(\omega_2)}, \quad &p=\infty,
\end{cases}
\end{align*}
where $L^{q,\infty}(\omega)$ denotes the weak weighted Lebesgue space.
\end{definition}

\begin{remark}
When $\omega_1=\omega_2=\omega,$ we denote $(L^q(\omega),L^p)^\alpha$ for short, which was introduced by Feuto \cite{F2014}. For $p=\infty,$ the space goes back to the weighted Morrey space defined by Komori and Shirai \cite{KS2009}. If $\omega\in A_\infty,$ then for $1\leq q_1\leq q_2\leq \alpha\leq p\leq \infty,$ we have $(L^{q_2}(\omega),L^p)^\alpha\subset(L^{q_1}(\omega),L^p)^\alpha$, and for $1\leq q\leq\alpha\leq p_1\leq p_2\leq\infty,$ we have $(L^q(\omega),L^{p_1})^\alpha\subset(L^q(\omega),L^{p_2})^\alpha.$
\end{remark}

In \cite{F2014},  Feuto proved that operators which are bounded on weighted
Lebesgue spaces and satisfy some local pointwise control, are bounded on weighted Amalgam spaces. His results included Calder\'on-Zygmund operators, Marcinkiewicz operators, maximal operators associated to Bochner-Riesz operators and their commutators. And his results also included Littlewood-Paley operators with rough kernels, whose control in this spaces was given by Wei and Tao in \cite{WT2013}. Readers can refer to \cite{CMP1999,DFS2013} for more research about the of operators on Amalgam space. In this paper, we will generalize Feuto's and Wei and Tao's results to multilinear type.

The theory of multilinear analysis related to the Calder\'on-Zygmund program originated in the work
of Coifman and Meyer \cite{CM11975,CM21978,CM31990}. Its study has been attracting a lot of attention in the last few decades.
A series of papers on this topic enriches this program, for example Christ and Journ\'e \cite{CJ1987},
Kenig and Stein \cite{KS1999}, and Grafakos and Torres \cite{GT12002,GT22002}. The authors \cite{LOPTT2009} introduced so-called multiple weights to develop the weighted multilnear Calder\'on-Zygmund theory and resolved some problems opened up in \cite{GT22002}.  Recently, many authors poured their much enthusiasm and devotion to the multilinear Littlewood-Paley theory for the unweighted case and weighted case, see \cite{CHO2014,CXY2015,GLMY2014,GO2012,H2012,XPY2015,XY2015} and so on.

In this paper, we give out a universal frame, which including multilinear Littlewood-Paley functions, multilinear Marcinkiewicz integral functions, and the multilinear Calder\'on-Zygmund operators, and multilinear singular integrals with nonsmooth kernels. We estimate the boundedness of these operator and their multilinear commutators and iterated commutators with the symbol $\vec b\in BMO^m$ on weighted amalgam spaces.

We also consider multilinear fractional type integral operators and their commutators' behaviors  on weighted amalgam space. In order to deal with the endpoint case, we introduce the amalgam-Campanato space and show that fractional integral integral operator are bounded operators from product amalgam space to amalgam-Campanato space.

This paper is organized as follows. In section 2, we prove that if the multilinear integral operator or sub-multilinear integral operator satisfy a point estimate (see following \eqref{AmalgamCondition} ), then they and their multilinear commutators and iterated commutators with the symbol $\vec b\in BMO^m$ are bounded on product weighted amalgam space. In section 3, we apply our main theorems in section 2 to some specific operators, including multilinear Littlewood-Paley functions with convolution type kernel and non-convolution type kernel, multilinear Marcinkiewicz integral functions with convolution type kernel and non-convolution type kernel and the multilinear Calder\'on-Zygmund operators and multilinear singular integrals with nonsmooth kernels. We will not state the corresponding results for the last two classes of multilinear operators since they satisfy our condition presented in some papers essentially, see \cite{AD2010,CW2012,DGGLY2009,DGY2010,GLY2011,GT12002,GT22002,LOPTT2009,PPTT2014} and so on. In the last section, we show that multilinear fractional type integral operator and its multilinear commutator and iterated commutator are bounded operators on product weighted amalgam space. Some endpoint estimate for multilinear fractional integral operator are also obtained in the non-weighed case, that is, we prove that multilinear fractional integral operators are bounded from product amalgam spaces to amalgam-Campanato space. It is deserved to be pointed out that even if in linear case, our results are also new.

Throughout the article, the constant $C$ always denotes a positive constant independent of the main variables, which may vary from line to line. For a ball $B\subset \mathbb{R}^n$ and $\lambda>0$, we use $\lambda B$ to denote the ball concentric with $B$ whose radius is $\lambda$~times of~$B's$. And we use $B^c$ to denote the supplementary of $B$ on $\mathbb{R}^n$, that is to say, $B^c=\mathbb{R}^n\backslash B$. As usual, $|E|$ denotes the Lebesgue measure of a measurable set $E$ and $\chi_E$  denotes the characteristic function of $E$. For $p>1$, we denote by $p'=p/(p-1)$ the conjugate exponent of $p$.

\section{Multilinear operators and their commutators on weighted amalgam space \label{s2}}
\hskip\parindent

We first recall the definitions of Muckenhoupt weights $A_p$, multiple weights $A_{\vec P}$ and \textrm{BMO} space.
\begin{definition}\label{weight}\cite{M1972}
Let $1\leq p<\infty$. Suppose that $\omega$ is a nonnegative function on $\mathbb{R}^n$. We say that $\omega\in A_{p}~(1<p<\infty)$ if it satisfies
\begin{align*}
\sup_B\bigg(\frac{1}{|B|}\int_B\omega(x)dx\bigg)^{1/p}\bigg(\frac{1}{|B|}\int_B\omega(x)^{1-p'}dx\bigg)^{1/{p'}}<\infty.
\end{align*}
A weight w belongs to the class $A_1$ if there exists a constant $C$ such that
$$\frac{1}{|B|}\int_B\omega(x)dx\leq C\inf_{y\in B}\omega(y).$$
\end{definition}
\begin{definition}\label{Multiplierweight}\cite{LOPTT2009}
Let $1\leq p_1,\ldots,p_m<\infty$ with $1/p=\sum_{j=1}^m1/{p_j}$. Suppose that $\vec{\omega}=(\omega_1,\ldots,\omega_m)$ and
each $¦Ø_j~(i = 1,\ldots,m)$ is a nonnegative function on $\mathbb{R}^n$. We say that $\vec{\omega}\in A_{\vec{p}}$ if it satisfies
$$
\sup_B\bigg(\frac{1}{|B|}\int_B\nu_{\vec{\omega}}(x)dx\bigg)^{1/p}\prod_{j=1}^m\bigg(\frac{1}{|B|}\int_B\omega_j(x)^{1-p_j'}dx\bigg)^{1/{p_j'}}<\infty
$$
where $\nu_{\vec{\omega}}=\prod_{j=1}^m\omega_j^{p/{p_j}}$. If $p_j=1$, $\bigg(\frac{1}{|B|}\int_B\omega_j^{1-p_j'}\bigg)^{1/{p_j'}}$ is understood as $(\inf_B\omega_j)^{-1}$.
\end{definition}
\begin{definition}\label{BMO}\cite{JN1961}
A local integral function $f$ is said to belong to $BMO^q(\mathbb{R}^n),q\geq1$ if
$$
\|f\|_{BMO^q}:=\sup_B\bigg(\frac{1}{|B|}\int_B|f(x)-f_B|^qdx\bigg)^\frac1q<\infty,
$$
where $f_B=\frac{1}{|B|}\int_Bf(x)dx$ denotes the mean value of $f$ over ball $B$.
\end{definition}
\begin{remark}
By John-Nirenberg's inequality, we have $\|f\|_{BMO^1}=\|f\|_{BMO^q}$ for all $q>1,$ so we denote by $BMO$ simple. From the definition, it could be seen that $|f_{2^kB}-f_B|\leq C k\|f\|_{BMO}.$
\end{remark}

And we also give out some auxiliary lemmas which will be used in following proof.
\begin{lemma}\label{ReHolder}\cite{GR1985}
If $\omega\in A_q$, for any  measurable subset $E$ of $B$, then there exists a constant $C$ and $\varepsilon>0$ such that
$$
\frac{\omega(E)}{\omega(B)}\leq C\bigg(\frac{|E|}{|B|}\bigg)^{\frac{\varepsilon}{1+\varepsilon}}.
$$
\end{lemma}
\begin{lemma}\label{doublemeasure}\cite{GR1985}
If $\omega\in A_q$, then the measure $\omega(x)dx$ is a doubling measure, that is to say, for any $\lambda>1$ and all balls $B$, there exists a constant $C>0$ such that
$$
\omega(\lambda B)\leq C \lambda^{nq}\omega(B).
$$
\end{lemma}
\begin{lemma}\label{MultiplierweightandWeight}\cite{WY2013}
Let $m\geq 2$, $p_1,\ldots,p_m\in(0,\infty)$ and $p\in(0,\infty)$ with $1/p=\sum_{j=1}^m1/{p_j}$. If $\omega_1,\ldots,\omega_m\in A_\infty$, then for any ball $B$, there exists a constants $C>0$ such that
$$
\prod_{j=1}^m\big(\omega_j(B)\big)^{p/{p_j}}\leq C\nu_{\vec{\omega}}(B).
$$
\end{lemma}
\begin{lemma}\label{Multipleweighteqvanlence}\cite{LOPTT2009}
Let $\vec\omega=(\omega_1,\ldots,\omega_m)$ and $1\leq p_1,\ldots, p_m<\infty.$ Then $\vec\omega\in A_{\vec P}$ if and only if
$$
\begin{cases}
\omega_j^{1-p_j'}\in A_{mp_j'},j=1\ldots,m,\\
\nu_{\vec\omega}\in A_{mp},
\end{cases}
$$
where the condition $\omega_j{1-p_j'}\in A_{mp_j'}$ in the case $p_j = 1$ is understood as $\omega_j^{1/m}\in A_1.$
\end{lemma}
\begin{lemma}\label{BMOandBMOw}\cite{MW1976}
Suppose $\omega\in A_\infty,$ then $\|b\|_{BMO(\omega)}\approx\|b\|_{BMO}.$ Here
$$
BMO^p(\omega)=\bigg\{b:\|b\|_{BMO^p(\omega)}=\sup_{B}\Big(\frac1{\omega(B)}\int_B|b(x)-b_{B,\omega}|^p\omega(x)dx\Big)^{1/p}<\infty\bigg\}
$$
and $b_{B,\omega}=\frac1{\omega(B)}\int_Bb(x)\omega(x)dx.$
\end{lemma}

From now on, we always adapt the following notation. Given a ball $B$ and $m$ functions $f_j,j=1,\cdots,m,$ we decompose $f_j=f^0_j+f_j^\infty$ with $f_j^0=f\chi_{2B}$ and $f_j^\infty=f\chi_{\mathbb R^n\backslash 2B}.$ Let $\mathcal{I}:=\{(d_1,\ldots,d_m)\in\{0,\infty\}^m: \text{there is at least one}~d_j\neq 0\},$ by $\vec f^{\vec D}$ we mean that $(f_1^{d_1},\ldots,f_m^{d_m})$ with $\vec D=(d_1,\cdots,d_m)\in \mathcal I.$ The nonempty set $\sigma\subset\{1,\cdots,m\}$ is the set of all number $i$ such that $d_i=\infty.$ and $\sigma^c$ denote by the supplementary of $\sigma.$ Now we can give out our main results in this section.

\begin{theorem}\label{OperatoronAmalgam}
Let $1\leq s\leq q_j \leq\alpha_j<p_j\leq\infty,\ j = 1,\ldots,m$ with
$1/q=\sum_{j=1}^m 1/{q_j}, 1/p=\sum_{j=1}^m 1/{p_j}, 1/{\alpha}=\sum_{j=1}^m 1/{\alpha_m}$
and  $p/{p_j}=q/{q_j}=\alpha/{\alpha_j}$, $\vec\omega\in A_{\vec{Q}/s}\cap (A_{\infty})^m$ and $m$-sublinear operator $\mathcal{T}$ satisfies: for any ball $B$ and almost everywhere $x\in B$
\begin{align}\label{AmalgamCondition}
\big|\mathcal{T}&(\vec{f}^{\vec{D}})(x)\big|\\
&\leq C_1\mathop{\prod}\limits_{j\in\sigma^c}\bigg(\frac{1}{|2B|}\int_{2B}|f_j(z)|^{s}dz\bigg)^{1/{s}}
\mathop{\sum}\limits_{k=1}^\infty\frac{k}{2^{nk|\sigma^c|}}\mathop{\prod}\limits_{j\in\sigma}\bigg(\frac{1}{2^{k+1}B}\int_{2^{k+1}B\backslash 2^kB}|f_j(z)|^{s}dz\bigg)^{1/{s}},\nonumber
\end{align}
Then
\begin{itemize}
\item[(i)] if $q_1,\ldots,q_m>1$ and $\mathcal{T}$ is bounded from $L^{q_1}(\omega_1)\times\cdots\times L^{q_m}(\omega_m)$ to $L^q(\nu_{\vec{\omega}})$, then
$\mathcal{T}$ is also bounded from $(L^{q_1}(\omega_1),L^{p_1})^{\alpha_1}\times\cdots\times (L^{q_m}(\omega_m),L^{p_m})^{\alpha_m}$ to $(L^q(\nu_{\vec{\omega}}),L^p)^{\alpha};$
\item[(ii)] if there exists a $q_j=1$ and $\mathcal{T}$ is bounded from $L^{q_1}(\omega_1)\times\cdots\times L^{q_m}(\omega_m)$ to $L^{q,\infty}(\nu_{\vec{\omega}})$, then
$\mathcal{T}$ is also bounded from $(L^{q_1}(\omega_1),L^{p_1})^{\alpha_1}\times\cdots\times (L^{q_m}(\omega_m),L^{p_m})^{\alpha_m}$ to $(L^{q,\infty}(\nu_{\vec{\omega}}),L^p)^{\alpha}.$
\end{itemize}
\end{theorem}

\begin{proof}
Let $1\leq s\leq q_j\leq \alpha_j<p_j$ and $f_j\in (L^{q_j}(\omega_j),L^{p_j})^{\alpha_j},j=1,\ldots,m.$ Fix $B:=B(y,r),$ then we have that for almost every $x\in B(y,r)$
\begin{align*}
|\mathcal{T}&(f_1,\ldots,f_m)(x)|\\
&\leq |\mathcal{T}(f_1^0,\ldots,f_m^0)(x)|+\mathop{\sum}\limits_{(d_1,\ldots,d_m)\in \mathcal{I}}|\mathcal{T}(f_1^{d_1},\ldots,f_m^{d_m})(x)|\\
&\leq|\mathcal{T}(f_1^0,\ldots,f_m^0)(x)|\\
\hspace*{12pt}&+C_1\mathop{\sum}\limits_{\sigma\neq\emptyset}\mathop{\prod}\limits_{j\in\sigma^c}\bigg(\frac{1}{|2B|}\int_{2B}|f_j(z)|^{s}dz\bigg)^{1/{s}}
\mathop{\sum}\limits_{k=1}^\infty\frac{k}{2^{nk|\sigma^c|}}\mathop{\prod}\limits_{j\in\sigma}\bigg(\frac{1}{2^{k+1}B}\int_{2^{k+1}B\backslash 2^kB}|f_j(z)|^{s}dz\bigg)^{1/{s}}.
\end{align*}
If $s=q_j$, we have
$$
\bigg(\frac{1}{|B|}\int_B|f_j(z)|^{s}dz\bigg)^{1/{s}}\leq \bigg(\int_B|f_j(z)|^{q_j}\omega_j(z)dz\bigg)^{1/{q_j}}|B|^{-\frac1{q_j}}(\inf_B\omega_j(z))^{-1/s}.
$$
And if $s<q_j,$ by H\"older's inequality, we derive
\begin{align*}
\bigg(\frac{1}{|B|}\int_B|f_j(z)|^{s}dz\bigg)^{1/{s}}\leq C\bigg(\int_B|f_j(z)|^{q_j}\omega_j(z)dz\bigg)^{1/{q_j}}(|B|)^{-\frac1{q_j}}
\bigg(\frac1{|B|}\int_B\omega_j(z)^{1-(\frac{q_j}{s})'}dz\bigg)^{\frac1{s(\frac{q_j}{s})'}}.
\end{align*}
It comes out that
\begin{align}\label{Estimate21}
\mathop{\sum}\limits_{(d_1,\ldots,d_m)\in \mathcal{I}}&|\mathcal{T}(f_1^{d_1},\ldots,f_m^{d_m})(x)|\\
&\leq C\sum_{\sigma\neq\emptyset}\prod\limits_{j\in\sigma^c}\frac{1}{|2B|^{1/{q_j}}}\bigg(\frac1{|2B|}\int_{2B}\omega_j(z)^{1-(\frac{q_j}{s})'}dz\bigg)^{\frac1{s(\frac{q_j}{s})'}}\|f\chi_{2B}\|_{L^{q_j}(\omega_j)}\nonumber\\
&\hspace*{12pt}\times\sum_{k=1}^\infty\frac{k}{2^{nk|\sigma^c|}} \prod\limits_{j\in\sigma}
\frac{1}{|2^{k+1}B|^{1/{q_j}}}\bigg(\frac1{|2^{k+1}B|}\int_{2^{k+1}B}\omega_j(z)^{1-(\frac{q_j}{s})'}dz\bigg)^{\frac1{s(\frac{q_j}{s})'}}\|f\chi_{2^{k+1}B}\|_{L^{q_j}(\omega_j)}\nonumber\\
&\leq C\sum_{k=1}^\infty k|2^{k+1}B|^{-\frac1q}\prod_{j=1}^m\bigg(\frac1{|2^{k+1}B|}\int_{2^{k+1}B}\omega_j(z)^{1-(\frac{q_j}{s})'}dz\bigg)^{\frac1{s(\frac{q_j}{s})'}}\|f_j\chi_{2^{k+1}B}\|_{L^{q_j}(\omega_j)}\nonumber\\
&\leq C\sum_{k=1}^\infty k\nu_{\vec\omega}(2^{k+1}B)^{-\frac1q}\prod_{j=1}^m\|f_j\chi_{2^{k+1}B}\|_{L^{q_j}(\omega_j)}.\nonumber
\end{align}
By Lemma \ref{MultiplierweightandWeight} and the condition $q/q_j=p/p_j=\alpha/\alpha_j,$ we have an important inequality as follows
$$
\nu_{\vec\omega}(B)^{\frac1\alpha-\frac1q-\frac1p}\leq C\prod_{j=1}^m\omega_j(B)^{\frac1{\alpha_j}-\frac1{q_j}-\frac1{p_j}}.
$$
Next we first prove $(i).$ Assume that all $q_j>1,j=1,\ldots,m.$ Since $\mathcal{T}: L^{q_1}(\omega_1)\times\cdots\times L^{q_m}(\omega_m)\rightarrow L^q(\nu_{\vec{\omega}})$,  then Lemma \ref{doublemeasure}, \ref{Multipleweighteqvanlence} and \ref{ReHolder} lead us to that
\begin{align}\label{Estimate211}
(\nu_{\vec{\omega}}&(B(y,r)))^{1/\alpha-1/q-1/p}\|\mathcal{T}(\vec{f})\chi_{B(y,r)}\|_{L^q(\nu_{\vec{\omega}})}\\
&\leq C\prod_{j=1}^m\omega_j(B(y,2r))^{1/{\alpha_j}-1/{q_j}-1/{p_j}}\|f_j\chi_{B(y,2r)}\|_{L^{q_j}(\omega_j)}\nonumber\\
&\hspace*{12pt}+C\sum_{k=1}^\infty k\bigg(\frac{\nu_{\vec\omega}(B(y,r))}{\nu_{\vec\omega}(2^{k+1}B(y,r))}\bigg)^{\frac1\alpha-\frac1p}\prod_{i=1}^m \omega_j(B(y,2^{k+1}r))^{1/{\alpha_j}-1/{q_j}-1/{p_j}} \|f_j\chi_{2^{k+1}B}\|_{L^{q_j}(\omega_j)}\nonumber\\
&\leq C\bigg(\sum_{k=0}^\infty k2^{-kn\frac{mp}{s}}\bigg)\prod_{i=1}^m \omega_j(B(y,2^{k+1}r))^{1/{\alpha_j}-1/{q_j}-1/{p_j}} \|f_j\chi_{2^{k+1}B}\|_{L^{q_j}(\omega_j)},\nonumber
\end{align}
and hence by H\"older's inequality we have
\begin{align*}
_{r}\|\mathcal{T}&(\vec{f})\|_{(L^q(\nu_{\vec{\omega}}),L^p)^{\alpha}}\\
&\leq C\bigg(\sum_{k=0}^\infty k2^{-kn\frac{mp}{s}}\bigg)\bigg(\int_{\mathbb{R}^n}\bigg(\prod_{i=1}^m \omega_j(B(y,2^{k+1}r))^{1/{\alpha_j}-1/{q_j}-1/{p_j}} \|f_j\chi_{2^{k+1}B}\|_{L^{q_j}(\omega_j)}\bigg)^{p}dy\bigg)^{\frac1p}\\
&\leq C\prod_{i=1}^m \bigg(\int_{\mathbb{R}^n}\bigg(\omega_j(B(y,2^{k+1}r))^{1/{\alpha_j}-1/{q_j}-1/{p_j}} \|f_j\chi_{2^{k+1}B}\|_{L^{q_j}(\omega_j)}\bigg)^{p_j}dy\bigg)^{\frac{1}{p_j}}\\
&\leq C\prod_{j=1}^m\|f_j\|_{(L^{q_j}(\omega_j),L^{p_j})^{\alpha_j}}.
\end{align*}
Taking supremum over all $r>0$ on both side, we get
$$
\|\mathcal{T}(\vec{f})\|_{(L^q(\nu_{\vec{\omega}}),L^p)^{\alpha}}\leq C\prod_{j=1}^m\|f_j\|_{(L^{q_j}(\omega_j),L^{p_j})^{\alpha_j}}.
$$

We now prove part $(ii)$. Suppose there is at least one $q_j=1$.  Combining $\mathcal{T}: L^{q_1}(\omega_1)\times\cdots\times L^{q_m}(\omega_m)\rightarrow L^{q,\infty}(\nu_{\vec{\omega}})$, $\vec\omega\in A_{\vec Q/{s}},$ with the H\"older inequality, we get
\begin{align*}
\|\mathcal{T}(\vec{f})&\chi_{B(y,r)}\|_{L^{q,\infty}(\nu_{\vec{\omega}})}\\
&\leq C\prod_{j=1}^m\|f_j\chi_{2B}\|_{L^{q_j}(\omega_j)}+C\sum_{k=1}^\infty k\bigg(\frac{\nu_{\vec\omega}(B(y,r))}{\nu_{\vec\omega}(2^{k+1}B(y,r))}\bigg)^{\frac1q}\prod_{i=1}^m\|f_j\chi_{2^{k+1}B}\|_{L^{q_j}(\omega_j)}
\end{align*}
according to \eqref{Estimate21}. Proceeding along the lines of the proof of part $(i)$, we find that
$$
\|\mathcal{T}(\vec{f})\|_{(L^{q,\infty}(\nu_{\vec{\omega}}),L^p)^{\alpha}}\leq C\prod_{j=1}^m\|f_j\|_{(L^{q_j}(\omega_j),L^{p_j})^{\alpha_j}}.
$$
\end{proof}

Next, we consider the commutators of multilinear or submultilinear operator which is defined by
\begin{align*}
\mathcal T_{\Sigma\vec b}(\vec f)&=\sum_{j=1}^m[\mathcal T b_j(\vec f)-\mathcal T(f_1,\cdots,f_{j-1},b_jf_j,f_{j+1}\cdots,f_m)]\\
&:=\sum_{j=1}^m [b_j,\mathcal T(\vec f)]=\sum_{j=1}^m\mathcal T_{b_j}(\vec f) .
\end{align*}
We have the strong type boundedness of the commutators $\mathcal T_{\Sigma\vec b}$ over product weighted amalgam spaces.

\begin{theorem}\label{multilinearCommutatoronAmalgam}
Let $1\leq s< q_j \leq\alpha_j<p_j\leq\infty,\ j = 1,\ldots,m$ with
$1/q=\sum_{j=1}^m 1/{q_j}, 1/p=\sum_{j=1}^m 1/{p_j}, 1/{\alpha}=\sum_{j=1}^m 1/{\alpha_m}$
and  $p/{p_j}=q/{q_j}=\alpha/{\alpha_j}$, $\vec\omega\in A_{\vec{Q}/s}\cap (A_{\infty})^m$ and $m$-sublinear operator $\mathcal{T}$ satisfies condition (\ref{AmalgamCondition}) and admits a multilinear commutator $T_{\Sigma \vec{b}}$. If $\mathcal{T}_{\Sigma\vec{b}}$ is bounded from $L^{q_1}(\omega_1)\times\cdots\times L^{q_m}(\omega_m)$ to $L^q(\nu_{\vec{\omega}})$, then
$\mathcal{T}_{\Sigma\vec{b}}$ is also bounded from $(L^{q_1}(\omega_1),L^{p_1})^{\alpha_1}\times\cdots\times (L^{q_m}(\omega_m),L^{p_m})^{\alpha_m}$ to $(L^q(\nu_{\vec{\omega}}),L^p)^{\alpha}$.
\end{theorem}

\begin{proof}
For every $j\in\{1,\ldots,m\},$ we write $|\mathcal{T}_{b_j}(\vec{f})(x)|$ as
\begin{align*}
|\mathcal{T}_{b_j}(\vec{f})(x)|&\leq|\mathcal{T}_{b_j}(\vec{f}^{0})(x)|+|b_j(x)-(b_j)_B||\mathcal{T}(\vec{f})^{\vec D}(x)|+|\mathcal{T}(((b_j)_B-b_j)\vec{f}^{\vec D}(x)|\\
&:=I_{j,1}(x)+I_{j,2}(x)+I_{j,3}(x).
\end{align*}
Via the condition that $\mathcal{T}_{\Sigma\vec{b}}$ is bounded from $L^{q_1}(\omega_1)\times\cdots\times L^{q_m}(\omega_m)$ to $L^q(\nu_{\vec{\omega}}),$ it gets
\begin{equation}\label{Estimate22}
\|\mathcal{T}_{\Sigma\vec{b}}(\vec f^{0})\chi_{B(y,r)}\|_{L^q(\nu_{\vec{\omega}})}\leq C\sum_{j=1}^m\|b_j\|_{BMO}\prod_{j=1}^m\|f_j\chi_{B(y,2r)}\|_{L^{q_j}(\omega_j)}.
\end{equation}
For $I_{j,2}(x)$.  By (\ref{AmalgamCondition}), \eqref{Estimate21} and Lemma \ref{BMOandBMOw}, we then have
\begin{equation}\label{Estimate23}
\|I_{j,2}\chi_{B(y,r)}\|_{L^q(\nu_{\vec{\omega}})}
\leq C\|b_j\|_{BMO}\sum_{k=1}^\infty k\bigg(\frac{\nu_{\vec\omega}(B(y,r))}{\nu_{\vec\omega}(2^{k+1}B(y,r))}\bigg)^{\frac1q}\prod_{i=1}^m\|f_j\chi_{2^{k+1}B}\|_{L^{q_j}(\omega_j)}.
\end{equation}
Now it turns to $I_{j,3}(x)$. Noting that for $s<q_j$ and $\omega_j^{1-(\frac {q_j}s)'}\in A_{m\frac {q_j}s},j=1,\cdots,m$, by H\"older's inequality and Lemma \ref{BMOandBMOw}, we have
\begin{align}\label{Estimate25}
\bigg(\frac1{2^{k+1}B}&\int_{2^{k+1}B}|f_j(x)(b_j(x)-(b_j)_B)|^sdx\bigg)^{\frac1s}\\
&\leq Ck\|b_j\|_{BMO}|2^{k+1}B|^{-\frac1q}\bigg(\frac1{|2^{k+1}B|}\int_{2^{k+1}B}\omega_j(z)^{1-(\frac{q_j}{s})'}dz\bigg)^{\frac1{s(\frac{q_j}{s})'}}\|f_j\chi_{2^{k+1}B}\|_{L^{q_j}(\omega_j)}.\nonumber
\end{align}
Then by a similar estimate to \ref{Estimate21}, we then have that
\begin{equation}\label{Estimate24}
\|I_{j,3}\chi_{B(y,r)}\|_{L^q(\nu_{\vec{\omega}})}
\leq C\|b_j\|_{BMO}\sum_{k=1}^\infty k\bigg(\frac{\nu_{\vec\omega}(B(y,r))}{\nu_{\vec\omega}(2^{k+1}B(y,r))}\bigg)^{\frac1q}\prod_{i=1}^m\|f_j\chi_{2^{k+1}B}\|_{L^{q_j}(\omega_j)}.
\end{equation}
Combining \ref{Estimate22}, \ref{Estimate23} and \ref{Estimate24} with the same process of prove part $(i)$ in Theorem \ref{OperatoronAmalgam}, we also have
$$
\|\mathcal{T}_{b_j}(\vec{f})\|_{(L^{q}(\nu_{\vec{\omega}}),L^p)^{\alpha}}
\leq C\|b_j\|_{\rm {BMO}}\prod_{j=1}^m\|f_j\|_{(L^{q_j}(\omega_j),L^{p_j})^{\alpha_j}},
$$
and it completes the proof of the theorem.
\end{proof}

Our results also follow for commutators of iterated commutators defined by
$$
\mathcal T_{\Pi\vec b}(\vec f)=[b_1,[b_2,[\cdots[b_m,\mathcal T]\cdots]]](\vec f).
$$

\begin{theorem}\label{IteratedCommutatoronAmalgam}
Let $1\leq s< q_j \leq\alpha_j<p_j\leq\infty,\ j = 1,\ldots,m$ with
$1/q=\sum_{j=1}^m 1/{q_j}, 1/p=\sum_{j=1}^m 1/{p_j}, 1/{\alpha}=\sum_{j=1}^m 1/{\alpha_m}$
and  $p/{p_j}=q/{q_j}=\alpha/{\alpha_j}$, $\vec\omega\in A_{\vec{Q}/s}\cap (A_{\infty})^m$ and $m$-sublinear operator $\mathcal{T}$ satisfies condition (\ref{AmalgamCondition}) and admits a iterated commutator $T_{\Pi \vec{b}}$. If $\mathcal{T}_{\Pi\vec{b}}$ is bounded from $L^{q_1}(\omega_1)\times\cdots\times L^{q_m}(\omega_m)$ to $L^q(\nu_{\vec{\omega}})$, then
$\mathcal{T}_{\Pi\vec{b}}$ is also bounded from $(L^{q_1}(\omega_1),L^{p_1})^{\alpha_1}\times\cdots\times (L^{q_m}(\omega_m),L^{p_m})^{\alpha_m}$ to $(L^q(\nu_{\vec{\omega}}),L^p)^{\alpha}$.
\end{theorem}

\begin{proof}
Without loss of generalization, we only consider the case $m=2$. Write
\begin{align*}
|\mathcal{T}_{\Pi\vec{b}}(\vec{f})(x)|&\leq|\mathcal{T}_{\Pi\vec{b}}(f_1^0,f_2^0)(x)|+\sum_{(d_1,d_2)\in\mathcal I}\bigg\{|(b_1(x)-(b_1)_B)(b_2(x)-(b_2)_B)||\mathcal{T}(f_1^{d_1},f_2^{d_2})(x)|\\
&\hspace*{12pt}+|b_1(x)-(b_1)_B)||\mathcal{T}(f_1^{d_1},((b_2)_B-b_2)f_2^{d_2})(x)|
+|b_2(x)-(b_2)_B||\mathcal{T}(((b_1)_B-b_1)f_1^{d_1},f_2^{d_2})(x)|\\
&\hspace*{12pt}+|\mathcal{T}(((b_1)_B-b_1)f_1^{d_1},((b_2)_B-b_2)f_2^{d_2})(x)|\bigg\}\\
&:=II_1(x)+II_2(x)+II_3(x)+II_4(x)+II_5(x).
\end{align*}
Since $\mathcal{T}_{\Pi\vec{b}}$ is bounded from $L^{q_1}(\omega_1)\times\cdots\times L^{q_m}(\omega_m)$ to $L^q(\nu_{\vec{\omega}}),$ we have
$$
\|II_1\chi_{B(y,r)}\|_{L^q(\nu_{\vec{\omega}})}
\leq C\prod_{j=1}^2\|b_j\|_{BMO}\|\prod_{j=1}^2\|f_j\chi_{B(y,2r)}\|_{L^{q_j}(\omega_j)}.
$$
For $II_2(x).$ Using (\ref{AmalgamCondition}), H\"older's inequality and Lemma \ref{MultiplierweightandWeight}, we obtain that by
\begin{align*}
\|II_{2}(\cdot)&\chi_{B(y,r)}\|_{L^q(\nu_{\vec{\omega}})}\\
&\leq C\bigg(\int_B|b_1(x)-(b_1)_B|^{q}\omega_1(x)^{q/{q_1}}|b_2(x)-(b_2)_B|^{q}\omega_2(x)^{q/{q_2}}dx\bigg)^{1/q}\\
&\hspace*{12pt}\times\sum_{k=1}^\infty k\prod_{i=1}^2\bigg(\frac{1}{|2^{k+1}B|}\int_{2^{k+1}B}|f_i(z)|^{q_i}\omega_i(z)dz\bigg)^{1/{q_i}}
\bigg(\frac{1}{|2^{k+1}B|}\int_{2^{k+1}B}\omega_i(z)^{1-\big(\frac{q_i}{s_i}\big)'}\bigg)^{\frac{1}{s_i}-\frac{1}{q_i}}\\
&\leq C\prod_{i=1}^2\bigg(\int_B|b_i(x)-(b_i)_B|^{q_i}\omega_i(x)dx\bigg)^\frac1{q_i}\sum_{k=1}^\infty k\bigg(\frac{\nu_{\vec\omega}(B(y,r))}{\nu_{\vec\omega}(2^{k+1}B(y,r))}\bigg)^{\frac1q}\prod_{i=1}^m\|f_j\chi_{2^{k+1}B}\|_{L^{q_j}(\omega_j)}\\
&\leq C_1\prod_{i=1}^2\|b_i\|_{BMO}
\sum_{k=1}^\infty k\bigg(\frac{\nu_{\vec\omega}(B(y,r))}{\nu_{\vec\omega}(2^{k+1}B(y,r))}\bigg)^{\frac1q}\prod_{i=1}^m\|f_j\chi_{2^{k+1}B}\|_{L^{q_j}(\omega_j)}.
\end{align*}
Next we estimate $II_5(x)$. Via \eqref{AmalgamCondition} and \eqref{Estimate25}, we obtain
\begin{align*}
II_5(x)&\leq C\sum_{k=1}^\infty k\prod_{i=1}^2\bigg(\frac{1}{2^{k+1}B}\int_{2^{k+1}B}|(b_i(z)-(b_i)_B)f_i(z)|^{s}\bigg)^{\frac1s}\\
&\leq C\sum_{k=1}^\infty k^3|2^{k+1}B|^{-\frac1q}\prod_{i=1}^2
\|b_i\|_{BMO}\bigg(\frac1{|2^{k+1}B|}\int_{2^{k+1}B}\omega_i(z)^{1-(\frac{q_i}{s})'}dz\bigg)^{\frac1{s(\frac{q_i}{s})'}}\|f_i\chi_{2^{k+1}B}\|_{L^{q_i}(\omega_i)}\\
& \leq C\sum_{k=1}^\infty k^3\nu_{\vec \omega}(2^{k+1}B)^{-\frac1q}\prod_{i=1}^2
\|b_i\|_{BMO}\|f_i\chi_{2^{k+1}B}\|_{L^{q_i}(\omega_i)}
\end{align*}
We then have
\begin{align*}
\|II_{5}(\cdot)\chi_{B(y,r)}\|_{L^q(\nu_{\vec{\omega}})}\leq C_1\prod_{i=1}^2\|b_i\|_{BMO}
\sum_{k=1}^\infty k\bigg(\frac{\nu_{\vec\omega}(B(y,r))}{\nu_{\vec\omega}(2^{k+1}B(y,r))}\bigg)^{\frac1q}\prod_{i=1}^m\|f_j\chi_{2^{k+1}B}\|_{L^{q_j}(\omega_j)}.
\end{align*}
Since there is the same method to deal with the terms $II_3(x)$ and $II_4(x),$ we only address $II_3(x).$  First, we have a point estimate
\begin{align*}
II_3(x)&\leq C_1|b_1(x)-(b_1)_B|\\
&\hspace*{12pt}\times\sum_{k=1}^\infty k\bigg(\frac{1}{|2^{k+1}B|}\int_{2^{k+1}B}|f_1(z)|^{s}\bigg)^{\frac1s}\bigg(\frac{1}{|2^{k+1}B|}\int_{2^{k+1}B}\big|f_2(z)(b_2(z)-(b_2)_B)\big|^{s}\bigg)^{\frac1s}.
\end{align*}
Using H\"older's inequality and \eqref{Estimate25}, we get
\begin{align*}
\|II_{3}(\cdot)\chi_{B(y,r)}\|_{L^q(\nu_{\vec{\omega}})}\leq C_1\prod_{i=1}^2\|b_i\|_{BMO}
\sum_{k=1}^\infty k\bigg(\frac{\nu_{\vec\omega}(B(y,r))}{\nu_{\vec\omega}(2^{k+1}B(y,r))}\bigg)^{\frac1q}\prod_{i=1}^m\|f_j\chi_{2^{k+1}B}\|_{L^{q_j}(\omega_j)}.
\end{align*}
Then we can complete this theorem by a similar proof of part (i) of Theorem \ref{OperatoronAmalgam}.
\end{proof}




\section{Some applications \label{s3}}
\hskip\parindent
We will apply our theorems in section \ref{s2} to multilinear Littlewood-Paley functions and Marcinkiewicz functions of convolution and non-convolution type, respectively.

\subsection{ Multilinear Littlewood-Paley functions \label{s3.1}}
\hskip\parindent

\textbf{I: Convolution Type}
\hskip\parindent

In \cite{XPY2015}, Xue, Peng and Yabuta introduced the multilinear Littlewood-Paley kernel as follows.
\begin{definition}\label{multilinearstandardkernel}
we say that a function $\psi$ defined on on $(\mathbb R^n)^m$ is a multilinear standard kernel, if it satisfies the following two conditions:
\begin{itemize}
\item[(i)] Size condition: for any $(y_1,\cdots, y_m)\in (\mathbb R^n)^m,$ there is a constant $C_1> 0,$ such
that
\begin{align*}
|\psi(y_1,\cdots,y_m)|\leq \frac{C_1}{(1+\sum_{j=1}^m|y_j|)^{mn+\delta}},
\end{align*}
for some $\delta> 0.$

\item[(i)] Smoothness condition: there exists a constant $C_2> 0,$ such that for some $\gamma>0$
\begin{align*}
|\psi(y_1,\cdots,y_j+z,\cdots,y_m)-\psi(y_1,\cdots,y_j,\cdots,y_m)|\leq \frac{C_2|z|^\gamma}{(1+|y_1|+\cdots+|y_m|)^{mn+\delta+\gamma}},
\end{align*}
whenever $2|z|\leq\max_{j=1,\cdots,m}|y_j|$ and $(y_1,\cdots,y_m)\in(\mathbb R^n)^m$.
\end{itemize}
For any $\vec f=(f_1,\cdots,f_m)\in\mathcal S(\mathbb R^n)\times\cdots\times\mathcal S(\mathbb R^n)$ and $x\notin\bigcap_{j=1}^m\text{supp}f_j$, we can define the multilinear Littlewood-Paley g-function by
where
\begin{align}\label{MultilinearLittlewood-Paleyg}
g(\vec f)(x)=\bigg(\int_0^\infty|\psi_t*\vec f(x)|^2\frac{dt}{t}\bigg)^{\frac{1}{2}}
\end{align}
where $\psi$ is the multilinear standard kernel with $\psi_t(y_1,\cdots,y_m)=\frac{1}{t^{mn}}\psi(\frac{y_1}{t},\cdots,\frac{y_m}{t})$.
\end{definition}

As pointed out that the bilinear Littlewood-Paley $g$-function can be
rewritten as a 4-linear Fourier multiplier with symbol $m(\xi,\eta)=\int_0^\infty \hat{\psi}(t\xi)\hat\psi(t\eta)\frac{dt}t$ in \cite{XPY2015}, and by the results of Grafakos-Miyachi-Tomita \cite{GMT2013} and the assumption of the sufficient
smooth kernel, we know that the bilinear Littlewood-Paley $g$-function is bounded from $L^{q_1}\times L^{q_2}\rightarrow L^q$ with $1<q_1,q_2<\infty$ and $\frac1q=\frac1{q_1}+\frac1{q_2}.$ So, in this paper, we assume that $g$ can be extended to be a bounded operator
for some $1\leq q_1,\cdots,q_m\leq\infty $ with $\frac1q =\frac1{q_1}+\cdots+\frac1{q_m},$ that is
\begin{align*}
g:L^{q_1}\times\cdots\times L^{q_m}\rightarrow L^q.
\end{align*}
Xue, Peng and Yabuta \cite{XPY2015} obtained the boundedness of multilinear littlewood-Paley $g$-function on weighted Lebesgue space as follows.
\begin{lemma}\label{Littlewood-PaleygLebesgue}(\cite{XPY2015})
Assume that $g$ can be extended to a bounded operator for some $1\leq s_1,\ldots,s_m<\infty$ with $1/s=\sum_{j=1}^m1/{s}$. Let $1\leq q_j <\infty,\ j = 1,\ldots,m$ with $1/q=\sum_{j=1}^m 1/{q_j}$, $\vec\omega\in A_{\vec{Q}}$ and $\vec{b}\in BMO(\mathbb{R}^n)$,
then
\begin{itemize}
\item[(i)]
if $1\leq q_1,\ldots,q_m<\infty$, $g$ is a bounded operators from $L^{q_1}(\omega_1)\times\cdots\times L^{q_m}(\omega_m)$ to $L^q(\nu_{\vec{\omega}})$ or $L^{q,\infty}(\nu_{\vec{\omega}})$;
\item[(ii)]if $1< q_1,\ldots,q_m<\infty$, $g_{\Sigma\vec b}$ and $g_{\Pi\vec{b}}$ are a bounded operators from $L^{q_1}(\omega_m)\times\cdots\times L^{q_m}(\omega_m)$ to $L^q(\nu_{\vec{\omega}})$.
\end{itemize}
\end{lemma}

We get the boundedness of multilinear littlewood-Paley $g$-function on weighted Amalgam space as follows.
\begin{theorem}\label{Littlewood-PaleygAmalgam}
Let $1\leq s\leq q_j \leq\alpha_j<p_j\leq\infty,\ j = 1,\ldots,m$ with
$1/q=\sum_{j=1}^m 1/{q_j}, 1/p=\sum_{j=1}^m 1/{p_j}, 1/{\alpha}=\sum_{j=1}^m 1/{\alpha_m}$
and  $p/{p_j}=q/{q_j}=\alpha/{\alpha_j}$, $\vec{\omega}\in A_{\vec{Q}\cap (A_{\infty})^m}$ and $g$ is Littlewood-Paley $g$-function.
Then
\begin{itemize}
\item[(i)]
$g$ is a bounded operator from $(L^{q_1}(\omega_1),L^{p_1})^{\alpha_1}\times\cdots\times (L^{q_m}(\omega_m),L^{p_m})^{\alpha_m}$ to $(L^q(\nu_{\vec{\omega}}),L^p)^{\alpha}$ or $(L^{q,\infty}(\nu_{\vec{\omega}}),L^p)^{\alpha}$.
\item[(ii)] $g_{\Sigma\vec b}$ and $g_{\Pi\vec{b}}$ are bounded operators from $(L^{q_1}(\omega_1),L^{p_1})^{\alpha_1}\times\cdots\times (L^{q_m}(\omega_m),L^{p_m})^{\alpha_m}$ to $(L^q(\nu_{\vec{\omega}}),L^p)^{\alpha}.$
\end{itemize}
\end{theorem}

\begin{proof}
By Theorem \ref{OperatoronAmalgam}, \ref{multilinearCommutatoronAmalgam} and \ref{IteratedCommutatoronAmalgam}, it suffices to verify that $g(\vec{f})(x)$ satisfies the condition \eqref{AmalgamCondition}. For any ball $B=B(y,r)\subset R^n$, we slip $f_j=f_j^0+f_j^\infty:=f_j\chi_{2B}+f_j\chi_{(2B)^c},j=1,\cdots,m$. Applying the size condition of multilinear standard kernel and Minkowshi's inequality yields
\begin{align*}
|g(\vec{f}^D)(x)|
&\leq C\bigg(\int_0^\infty\bigg|\int_{(\mathbb R^n)^m}\frac{t^\delta}{(t+\sum_{j=1}^m|x-y_j|)^{mn+\delta}}\prod_{j=1}^m|f_{j}^{d_j}(y_j)|dy_j\bigg|^2\frac{dt}{t}\bigg)^{\frac12}\\
&\leq \int_{(\mathbb{R}^n)^m}\bigg(\int_0^\infty\frac{t^{2\delta}}{(t+\sum_{j=1}^m|x-y_j|)^{2mn+2\delta}}\frac{dt}{t}\bigg)^{\frac{1}{2}}
\prod_{j=1}^m|f_{j}^{d_j}(y_j)|d\vec{y},
\end{align*}
for $x\in B$.

Since $x\in B(y,r)$ and there is $d_j=\infty,$ then we have $|x-y_j|\geq 2^{k-1}r$ for any $y_j\in2^{k+1}B\backslash 2^kB$. Therefore
\begin{align*}
|g(\vec{f}^D)(x)|
&\leq C\prod_{j\in\sigma^c}\int_{2B}|f_{j}(y_j)|dy_j\prod_{j\in\sigma}\sum_{k=1}^\infty\int_{2^{k+1}B\backslash 2^kB}
\bigg(\int_0^\infty\frac{t^{2\delta}}{(t+2^{k-1}r)^{2(mn+\delta)}}\frac{dt}{t}\bigg)^{\frac12}\\
&\leq C\prod_{j\in\sigma^c}\int_{2B}|f_{j}(y_j)|dy_j\prod_{j\in\sigma}\sum_{k=1}^\infty\int_{2^{k+1}B\backslash 2^kB}|f_{j}(y_j)|dy_j
\bigg(\int_0^{2^{k-1}r}\frac{t^{2\delta-1}}{(2^{k-1}r)^{2(mn+\delta)}}dt\bigg)^{\frac12}\\
&\hspace*{12pt}+C\prod_{j\in\sigma^c}\int_{2B}|f_{j}(y_j)|dy_j\prod_{j\in\sigma}\sum_{k=1}^\infty\int_{2^{k+1}B\backslash 2^kB}|f_{j}(y_j)|dy_j
\bigg(\int_{2^{k-1}r}^\infty\frac{t^{2\delta-1}}{t^{2(mn+\delta)}}dt\bigg)^{\frac12}\\
&\leq C\prod_{j\in\sigma^c}\int_{2B}|f_{j}(y_j)|dy_j\prod_{j\in\sigma}\sum_{k=1}^\infty\frac{1}{(2^kr)^{mn}}\int_{2^{k+1}B\backslash 2^kB}|f_{j}(y_j)|dy_j\\
&\leq C\prod_{j\in\sigma^c}\frac{1}{|2B|}\int_{2B}|f_{j}(y_j)|dy_j\prod_{j\in\sigma}
\sum_{k=1}^\infty\frac{1}{2^{kn|\sigma^c|}}\frac{1}{|2^{k+1}B|}\int_{2^{k+1}B\backslash 2^kB}|f_{j}(y_j)|dy_j,
\end{align*}
which finish the proof of Theorem \ref{Littlewood-PaleygAmalgam}.
\end{proof}

In \cite{CXY2015}, Chen, Xue and Yabuta introduced the corresponding  multilinear area integral $S$ of Lusin  as follows:
\begin{align}\label{areaintegral}
S(\vec f)(x)=\bigg(\int\int_{\Gamma_\eta(x,t)}|\psi_t*\vec f(z)|^2\frac{dzdt}{t^{n+1}}\bigg)^{\frac{1}{2}},
\end{align}
where $\Gamma_\eta(x,t)=\{(z,t)\in\mathbb{R}^{n+1}:|z-x|<\eta t\}$ is a cone with vertex $x$ for some fixed $\eta>0$.
\begin{lemma}\label{areaintegralLebesgue}\cite{CXY2015}
Assume that $S$ can be extended to a bounded operator for some $1\leq s_1,\ldots,s_m<\infty$ with $1/s=\sum_{j=1}^m1/{s}$.
Let $1\leq q_j <\infty,\ j = 1,\ldots,m$ with
$1/q=\sum_{j=1}^m 1/{q_j}$ and $\vec\omega\in A_{\vec{Q}}$ ,
then $S$ is a bounded operators from $L^{q_1}(\omega_1)\times\cdots\times L^{q_m}(\omega_m)$ to $L^q(\nu_{\vec{\omega}})$ or $L^{q,\infty}(\nu_{\vec{\omega}})$.
\end{lemma}
By the result above, we have following results for the multilinear area integral of Lusin $S$. Since the proof is very similar with it of Theorem \ref{Littlewood-PaleygAmalgam}, we omit it.
\begin{theorem}\label{areaintegralAmalgam}
Let $1\leq s\leq q_j \leq\alpha_j<p_j\leq\infty,\ j = 1,\ldots,m$ with
$1/q=\sum_{j=1}^m 1/{q_j}, 1/p=\sum_{j=1}^m 1/{p_j}, 1/{\alpha}=\sum_{j=1}^m 1/{\alpha_m}$
and  $p/{p_j}=q/{q_j}=\alpha/{\alpha_j}$, $\vec{\omega}\in A_{\vec{Q}\cap (A_{\infty})^m}$.
Then $S$ is a bounded operator from $(L^{q_1}(\omega_1),L^{p_1})^{\alpha_1}\times\cdots\times (L^{q_m}(\omega_m),L^{p_m})^{\alpha_m}$ to $(L^q(\nu_{\vec{\omega}}),L^p)^{\alpha}$ or $(L^{q,\infty}(\nu_{\vec{\omega}}),L^p)^{\alpha}$.
\end{theorem}
The corresponding multilinear Littlewood-paley $g_{\lambda}^*$ function introduced by Shi, Xue and Yabuta in \cite{SXY2014} is defined by following equation
\begin{align}\label{Littlewood-Paleyglambda}
g_{\lambda}^*(\vec f)(x)=\bigg(\int\int_{\mathbb{R}^{n+1}}\bigg(\frac{t}{t+|x-z|}\bigg)^{\lambda n}|\psi_t*\vec f(z)|^2\frac{dzdt}{t^{n+1}}\bigg)^{\frac{1}{2}},
\end{align}
where $\lambda>1.$
\begin{lemma}\label{Littlewood-PaleyglambdaLebesgue}\cite{SXY2014}
Suppose that for some $1\leq s_1,\ldots,s_m<\infty$ with $1/s=\sum_{j=1}^m1/{s}$, $g_\lambda^*$ is bounded from $L^{s_1}(\mathbb{R}^n)\times\cdots\times L^{s_m}(\mathbb{R}^n)$. Let $\lambda>2m$, $0<\gamma<\min\{n(\lambda-2m)/2,\delta\}$, $1\leq q_j <\infty,\ j = 1,\ldots,m$ with
$1/q=\sum_{j=1}^m 1/{q_j}$ and $\vec\omega\in A_{\vec{Q}}$,
then $g_{\lambda}^*$ is a bounded operators from $L^{q_1}(\omega_1)\times\cdots\times L^{q_m}(\omega_m)$ to $L^q(\nu_{\vec{\omega}})$ or $L^{q,\infty}(\nu_{\vec{\omega}})$.
\end{lemma}
We have the corresponding result for $g_{\lambda}^*$ on weighted Amalgam space.
\begin{theorem}\label{Littlewood-PaleyglambdaAmalgam}
Let $\lambda>2m+1$, $0<\gamma<\min\{n(\lambda-2m)/2,\delta\}$, $1\leq s\leq q_j \leq\alpha_j<p_j\leq\infty,\ j = 1,\ldots,m$ with
$1/q=\sum_{j=1}^m 1/{q_j}, 1/p=\sum_{j=1}^m 1/{p_j}, 1/{\alpha}=\sum_{j=1}^m 1/{\alpha_m}$
and  $p/{p_j}=q/{q_j}=\alpha/{\alpha_j}$. If $\vec{\omega}\in A_{\vec{Q}\cap (A_{\infty})^m}$,
then $g_{\lambda}^*$ is a bounded operator from $(L^{q_1}(\omega_1),L^{p_1})^{\alpha_1}\times\cdots\times (L^{q_m}(\omega_m),L^{p_m})^{\alpha_m}$ to $(L^q(\nu_{\vec{\omega}}),L^p)^{\alpha}$ or $(L^{q,\infty}(\nu_{\vec{\omega}}),L^p)^{\alpha}$.
\end{theorem}
\begin{proof}
In fact, we can decompose $g_{\lambda}^*(\vec{f})(x)$ into the sum of multilinear area integral of Lusin $S_{2^j}(\vec{f})(x)$ as follows
\begin{align*}
g_{\lambda}^*(\vec f)(x)&\leq\bigg(\int_0^\infty\int_{\Gamma(x)}\bigg(\frac{t}{t+|x-z|}\bigg)^{\lambda n}|\psi_t*\vec f(z)|^2\frac{dzdt}{t^{n+1}}\bigg)^{\frac{1}{2}}\\
&\hspace*{12pt}+\sum_{j=1}^\infty\bigg(\int_0^\infty\int_{2^{j-1}t\leq|z-x|<2^jt}\bigg(\frac{t}{t+|x-z|}\bigg)^{\lambda n}|\psi_t*\vec f(z)|^2\frac{dzdt}{t^{n+1}}\bigg)^{\frac{1}{2}}\\
&\leq C\sum_{j=0}^\infty\frac{1}{2^{j\lambda n/2}}S_{2^j}(\vec{f})(x).
\end{align*}
The remainder statement is the same and we omit it.
\end{proof}

\textbf{II: Nonconvolution Type}
\hskip\parindent

Next, we consider the multilinear Littlewood-Paley functions with the kernel of nonconvolution type introduced by Xue and Yan \cite{XY2015}.
\begin{definition}(\textbf{Integral smooth condition of C-Z type I})\label{C-ZtypeI}
For any $t\in(0,\infty)$,
let $K_t(x,y_1,\ldots,y_m)$ be a locally integrable function defined away from the diagonal
$x=y_1=\cdots=y_m$ in $(\mathbb{R}^n)^{m+1}$ and denote $(x, \vec{y})=(x,y_1,\ldots,y_m)$. We say $K_t$ satisfies
the integral condition of C-Z type I, if for some positive constants $\gamma, A$, and $B > 1$,
the following inequalities hold:
\begin{align*}
\bigg(\int_0^\infty|K_t(x,\vec{y})|^2\frac{dt}{t}\bigg)^{\frac{1}{2}}\leq\frac{A}{\bigg(\sum_{j=1}^m|x-y_j|\bigg)^{mn}},
\end{align*}
\begin{align*}
\bigg(\int_0^\infty|K_t(z,\vec{y})-K_t(x,\vec{y})|^2\frac{dt}{t}\bigg)^{\frac{1}{2}}\leq\frac{A|z-x|^{\gamma}}{\bigg(\sum_{j=1}^m|x-y_j|\bigg)^{mn+\gamma}},
\end{align*}
whenever $|x-z|\leq\max_{j=1}^m|x-y_j|/B$;
and
\begin{align*}
\bigg(\int_0^\infty|K_t(z,y_1,\ldots,y_j,\ldots,y_m)-K_t(x,y_1,\ldots,y_j',\ldots,y_m)|^2\frac{dt}{t}\bigg)^{\frac{1}{2}}
\leq\frac{A|y_j-y_j'|^{\gamma}}{\bigg(\sum_{j=1}^m|x-y_j|\bigg)^{mn+\gamma}},
\end{align*}
whenever $|y_j-y_j'|\leq |x-y_j|/B$, for $j=1,\ldots,m$. The multilinear square function $T$ is defined by
\begin{align}\label{squarefunctionI}
T(\vec{f})(x):=\Bigg(\int_0^\infty\bigg|\int_{(\mathbb{R}^n)^m}K_t(x,y_1,\ldots,y_m)\prod_{j=1}^mf_j(y_j)dy_1\cdots dy_m\bigg|^2\frac{dt}{t}\Bigg)^{1/2}
\end{align}
for any $\vec{f}=(f_1,\ldots,f_m)\in \mathcal{S}(\mathbb{R}^n)\times\cdots\times\mathcal{S}(\mathbb{R}^n)$ and all $x\notin \bigcap_{j=1}^m\text{supp} f_j$.
\end{definition}

\begin{definition}(\textbf{Integral smooth condition of C-Z type II})\label{C-ZtypeII}
 For any $t\in(0,\infty)$,
let $K_t(x,y_1,\ldots,y_m)$ be a locally integrable function defined away from the diagonal
$x=y_1=\cdots=y_m$ in $(\mathbb{R}^n)^{m+1}$ and denote $(x, \vec{y})=(x,y_1,\ldots,y_m)$. We say $K_t$ satisfies
the integral condition of C-Z type II, if for some positive constants $\gamma, A$, and $B > 1$,
the following inequalities hold:
\begin{align*}
\bigg(\int\int_{\mathbb{R}_+^{n+1}}\bigg(\frac{t}{t+|z|}\bigg)^{\lambda n}
|K_t(x-z,\vec{y})|^2\frac{dzdt}{t^{n+1}}\bigg)^{\frac{1}{2}}\leq\frac{A}{\bigg(\sum_{j=1}^m|x-y_j|\bigg)^{mn}},
\end{align*}
\begin{align*}
\bigg(\int\int_{\mathbb{R}_+^{n+1}}\bigg(\frac{t}{t+|z|}\bigg)^{\lambda n}
|K_t(x-z,\vec{y})-K_t(x'-z,\vec{y})|^2\frac{dzdt}{t^{n+1}}\bigg)^{\frac{1}{2}}\leq\frac{A|x-x'|^\gamma}{\bigg(\sum_{j=1}^m|x-y_j|\bigg)^{mn+\gamma}},
\end{align*}
whenever $|x-x'|\leq\max_{j=1}^m|x-y_j|/B$;
and
\begin{align*}
\bigg(\int\int_{\mathbb{R}_+^{n+1}}&\bigg(\frac{t}{t+|z|}\bigg)^{\lambda n}
|K_t(x-z,y_1,\ldots,y_j,\ldots,y_m)-K_t(x-z,y_1,\ldots,y_j',\ldots,y_m)|^2\frac{dzdt}{t^{n+1}}\bigg)^{\frac{1}{2}}\\
&\leq\frac{A|y_j-y_j'|^\gamma}{\bigg(\sum_{j=1}^m|x-y_j|\bigg)^{mn+\gamma}},
\end{align*}
whenever $|y_j-y_j'|\leq |x-y_j|/B$, for $j=1,\ldots,m$. The multilinear square function $T_\lambda$ is defined by
\begin{align}\label{squarefunctionII}
T_\lambda(\vec{f})(x):=\Bigg(\int\int_{\mathbb{R}_+^{n+1}}\bigg(\frac{t}{t+|x-z|}\bigg)^{\lambda n}
\bigg|\int_{(\mathbb{R}^n)^m}K_t(x,y_1,\ldots,y_m)\prod_{j=1}^mf_j(y_j)dy_1\cdots dy_m\bigg|^2\frac{dzdt}{t}\Bigg)^{1/2}
\end{align}
for any $\vec{f}=(f_1,\ldots,f_m)\in \mathcal{S}(\mathbb{R}^n)\times\cdots\times\mathcal{S}(\mathbb{R}^n)$ and all $x\notin \bigcap_{j=1}^m\text{supp} f_j$.
\end{definition}

\begin{definition}(\textbf{Multilinear Littlewood-Paley kernel})
Let $K(x,\vec{y})$ be a local integral function defined away from the diagonal $x=y_1=\cdots y_m$ in $(\mathbb{R}^n)^{m+1}$. $K$ is called a a multilinear Littlewood-Paley kernel if for some positive constants $A$, $\gamma_0$, $\delta$ and $B_1>1$,
\begin{align*}
|K(x,\vec{y})|\leq\frac{A}{\bigg(1+\sum_{j=1}^m|x-y_j|\bigg)^{mn+\delta}},
\end{align*}
\begin{align*}
|K(z,\vec{y})-K(x,\vec{y})|\leq\frac{A|z-x|^{\gamma_0}}{\bigg(1+\sum_{j=1}^m|x-y_j|\bigg)^{mn+\delta+\gamma_0}},
\end{align*}
whenever $|x-z|\leq\max_{j=1}^m|x-y_j|/{B_1}$;
and
\begin{align*}
|K(z,y_1,\ldots,y_j,\ldots,y_m)-K(x,y_1,\ldots,y_j',\ldots,y_m)|
\leq\frac{A|y_j-y_j'|^{\gamma_0}}{\bigg(\sum_{j=1}^m|x-y_j|\bigg)^{mn+\delta+\gamma_0}},
\end{align*}
whenever $|y_j-y_j'|\leq |x-y_j|/{B_1}$ for $j=1,\ldots,m$.
\end{definition}

Given a kernel $K$, denote $K_t(x,y_1,\ldots,y_m)=\frac{1}{t^{mn}}K(\frac{x}{t},\frac{y_1}{t},\ldots,\frac{y_m}{t})$. In \cite{XY2015}, Xue and Yan obtained the weighted strong and weak type estimates for the three multilinear operators on weighted Lebesgue spaces with a thorough analysis. We points out here that $T$ and $T_\lambda$ associated with the multilinear Littlewood-Paley kernel are called multilinear Littlewood¨CPaley $g$ and $g_\lambda^*$ functions respectively. If $K$ is a multilinear Littlewood-Paley kernel, then $K_t$ satisfies the integral smooth condition of C-Z type I with $0 <\gamma\leq\min\{\delta, \gamma_0\}$, and if more $\lambda>2m$ it also satisfies the integral smooth condition of C-Z type II with $0 <\gamma\leq\min\{(\lambda_n-2mn)/2, \gamma_0, n/2\}$. We state their theorems as our lemmas.

\begin{lemma}\label{squarefunctionILebesgue}(\cite{XY2015})
Let $T$ be the operator defined in Definition \ref{C-ZtypeI} with the kernel satisfying the
integral condition of C-Z type I and assume that $T$ can be extended to a bounded operator for some $1\leq s_1,\ldots,s_m<\infty$ with $1/s=\sum_{j=1}^m1/{s}$. Let $1/q=\sum_{j=1}^m1/{q_j}$ with $1\leq q_1,\ldots,q_m<\infty$, and assume that $\vec{\omega}\in A_{\vec{Q}}$, then $T$ is a bounded operator from $L^{q_1}(\omega_1)\times\cdots\times L^{q_m}(\omega_m)$ to $L^q(\nu_{\vec{\omega}})$ or $L^{q,\infty}(\nu_{\vec{\omega}})$.
\end{lemma}

\begin{lemma}\label{squarefunctionIILebesgue}(\cite{XY2015})
Let $T_\lambda$ be the operator defined in Definition \ref{C-ZtypeII} with the kernel satisfying the integral condition of C-Z type II with $\lambda>2m$ and assume that $T_\lambda$ can be extended to a bounded operator for some $1\leq s_1,\ldots,s_m<\infty$ with $1/s=\sum_{j=1}^m1/{s}$. Let $1/p=\sum_{j=1}^m1/{p_j}$
with $1\leq p_1,\ldots,p_m<\infty$ and $\lambda>2m$, and assume that $\vec{\omega}\in A_{\vec{p}}$, then $T_\lambda$ is a bounded operator from $L^{p_1}(\omega_1)\times\cdots\times L^{p_m}(\omega_m)$ to $L^p(\nu_{\vec{\omega}})$ or $L^{p,\infty}(\nu_{\vec{\omega}})$.
\end{lemma}

The multilinear square function $T$ and $T_\lambda$ have the same properties on weighted Amalgam space too.

\begin{theorem}\label{squarefunctionIAmalgam}
Let $T$ be the operator defined in Definition \ref{C-ZtypeI} with the kernel satisfying the
integral condition of C-Z type I and assume that $T$ can be extended to a bounded operator for some $1\leq r_1,\ldots,r_m<\infty$ with $1/r=\sum_{j=1}^m1/{r}$.
Let $1\leq s\leq q_j \leq\alpha_j<p_j\leq\infty,\ j = 1,\ldots,m$ with
$1/q=\sum_{j=1}^m 1/{q_j}, 1/p=\sum_{j=1}^m 1/{p_j}, 1/{\alpha}=\sum_{j=1}^m 1/{\alpha_m}$
and  $p/{p_j}=q/{q_j}=\alpha/{\alpha_j}$, $\vec{\omega}\in A_{\vec{Q}\cap (A_{\infty})^m}$, then
$T$ is bounded from $(L^{q_1}(\omega_1),L^{p_1})^{\alpha_1}\times\cdots\times (L^{q_m}(\omega_m),L^{p_m})^{\alpha_m}$ to $(L^q(\nu_{\vec{\omega}}),L^p)^{\alpha}$ or $(L^{q,\infty}(\nu_{\vec{\omega}}),L^p)^{\alpha}$.
\end{theorem}

\begin{proof}
By Minkowshi's inequality and the condition of C-Z type I, we get
\begin{align*}
|T(\vec{f}^D)(x)|
&\leq \int_{(\mathbb R^n)^m}\bigg(\int_0^\infty|K_t(x,y_1,\cdots,y_m)|^2\frac{dt}{t}\bigg)^\frac12\prod_{j=1}^m|f^{d_j}_j(y_j)|dy_1\cdots dy_m\\
&\leq \int_{(\mathbb R^n)^m}\frac{A}{\bigg(\sum_{j=1}^m|x-y_j|\bigg)^{mn}}\prod_{j=1}^m|f^{d_j}_j(y_j)|dy_1\cdots dy_m\\
&\leq C\prod_{j\in\sigma^c}\frac{1}{|2B|}\int_{2B}|f_j(y_j)|dy_j
\prod_{j\in\sigma}\sum_{k=1}^\infty\frac{1}{2^{kn|\sigma^c|}}\frac{1}{|2^{k+1}B|}\int_{2^{k+1}B\backslash 2^kB}|f_j(y_j)|dy_j,
\end{align*}
which shows that $T(\vec{f}^D)(x)$ satisfies condition (\ref{AmalgamCondition}).
\end{proof}

\begin{theorem}\label{squarefunctionIIAmalgam}
Let $T_\lambda$ be the operator defined in Definition \ref{C-ZtypeII} with the kernel satisfying the
integral condition of C-Z type II with $\lambda>2m$ and assume that $T_\lambda$ can be extended to a bounded operator for some $1\leq r_1,\ldots,r_m<\infty$ with $1/r=\sum_{j=1}^m1/{r}$..
Let $1\leq s\leq q_j \leq\alpha_j<p_j\leq\infty,\ j = 1,\ldots,m$ with
$1/q=\sum_{j=1}^m 1/{q_j}, 1/p=\sum_{j=1}^m 1/{p_j}, 1/{\alpha}=\sum_{j=1}^m 1/{\alpha_m}$
and  $p/{p_j}=q/{q_j}=\alpha/{\alpha_j}$, $\vec{\omega}\in A_{\vec{Q}\cap (A_{\infty})^m}$, then
$T_\lambda$ is a bounded operator from $(L^{q_1}(\omega_1),L^{p_1})^{\alpha_1}\times\cdots\times (L^{q_m}(\omega_m),L^{p_m})^{\alpha_m}$ to $(L^q(\nu_{\vec{\omega}}),L^p)^{\alpha}$ or $(L^{q,\infty}(\nu_{\vec{\omega}}),L^p)^{\alpha}$.
\end{theorem}

\begin{proof}
By Minkowshi's inequality and the condition of C-Z type II, we obtain
\begin{align*}
|T_\lambda(\vec{f}^D)(x)|
&\leq \int_{(\mathbb R^n)^m}\bigg(\int\int_{\mathbb{R}_+^{n+1}}\bigg(\frac{t}{t+|z|}\bigg)^{\lambda n}|K_t(x-z,\vec{y})|^2\frac{dzdt}{t^{n+1}}\bigg)^{\frac{1}{2}}\prod_{j=1}^m|f^{d_j}_j(y_j)|dy_1\cdots dy_m\\
&\leq \int_{(\mathbb R^n)^m}\frac{A}{\bigg(\sum_{j=1}^m|x-y_j|\bigg)^{mn}}\prod_{j=1}^m|f^{d_j}_j(y_j)|dy_1\cdots dy_m\\
&\leq C\prod_{j\in\sigma^c}\frac{1}{|2B|}\int_{2B}|f_j(y_j)|dy_j
\prod_{j\in\sigma}\sum_{k=1}^\infty\frac{1}{2^{kn|\sigma^c|}}\frac{1}{|2^{k+1}B|}\int_{2^{k+1}B\backslash 2^kB}|f_j(y_j)|dy_j,
\end{align*}
which proves that $T_\lambda(\vec{f}^D)(x)$ satisfies condition (\ref{AmalgamCondition}).
\end{proof}

\subsection{Multilinear Marcinkiewicz functions \label{s3.2}}
\hskip\parindent

\textbf{I: Convolution Type}
\hskip\parindent

Based on the Coifman-Meyer multilinear singular integral $T(f_1,\ldots, f_m)(x)=(K*(f_1\otimes\cdots\otimes f_m)(x)$ and Stein's work in \cite{S1958}, Chen, Xue and Yabuta \cite{CXY2015} introduced the multilinear Marcinkiewicz integral as follows.
\begin{definition}\textbf{Multilinear Marcinkiewicz Integral}
Let $\Omega$ be a function defined on $(\mathbb{R}^n)^m$ with the following properties:
\begin{itemize}
\item[(i)]$\Omega$ is homogeneous of degree 0, i.e.,
$\Omega(\ell\vec{y})=\Omega(\vec{y})$ for any $\ell>0$ and $\vec{y}=(y_1, \ldots,y_m)\in(\mathbb{R}^n)^m$;
\item[(ii)]$\Omega$ is Lipschitz continuous on $(\mathbb{S}^{n-1})^m$, i.e., there are $\beta\in(0,1)$ and $C>0$ such that for any $\vec{\theta}=(\eta_1,\ldots,\eta_m), \vec{\xi}=(\xi_1,\ldots,\xi_m)\in(\mathbb{R}^n)^m$,
$$|\Omega(\vec{\eta})-\Omega(\vec{\xi})|\leq C|\vec{\theta'}-\vec{\xi'}|,$$
where $(\eta_1, \ldots, \eta_m)'=\frac{(\eta_1,\ldots,\eta_m)}{|\eta_1|+\cdots+|\eta_m|}$;
\item[(iii)]The The integral of $\Omega$ on $(B(0, 1))^m$ vanishes,
$$ \int_{B(0,1)^m}\frac{\Omega(\vec{y})}{|\vec{y}|^{m(n-1)}}d\vec{y}=0.$$
\end{itemize}
For any $\vec{f}=(f_1,\ldots,f_m)\in \mathcal{S}(\mathbb{R}^n)\times\cdots\times\mathcal{S}(\mathbb{R}^n)$, the multilinear Marcinkiewicz integral of convolution type $\mu$ is defined by
\begin{align}\label{ConvolutedMarcinkiewicz}
\mu(\vec{f})(x):=\bigg(\int_0^\infty
\bigg|\frac{1}{t^m}\int_{(\mathbb{R}^n)^m}\frac{\Omega(x-y_1,\ldots,x-y_m)\chi_{(B(x,t))^m}}{|(x-y_1,\ldots,x-y_m)|^{m(n-1)}}\prod_{j=1}^mf_j(y_j)d\vec{y}\bigg|^2\frac{dt}{t}\bigg)^{1/2}
\end{align}
\end{definition}

Chen, Xue and Yabuta utilized real variable techniques to show that $\mu$ is a bounded operator on weighted Lebesgue space.
\begin{lemma}\label{ConvolutionMarcinkiewiczLebesgue}\cite{CXY2015}
Assume that $\mu$ can be extended to a bounded operator for some $1\leq \theta_1,\ldots, \theta_m$ with
$1/{\theta}=\sum_{j=1}^m1/{\theta_j}<\infty$. Suppose $\vec\omega\in A_{\vec{Q}}$ and $1/q=\sum_{j=1}^m 1/{q_j}$ with $1\leq q_j <\infty,\ j = 1,\ldots,m$,
then $\mu$ is a bounded operators from $L^{q_1}(\omega_1)\times\cdots\times L^{q_m}(\omega_m)$ to $L^q(\nu_{\vec{\omega}})$ or $L^{q,\infty}(\nu_{\vec{\omega}})$.
\end{lemma}

Therefore we consider the behavior of multilinear Mincinkiewicz $\mu$ on weighted Amalgam space.
\begin{theorem}\label{ConvolutionMarcinkiewiczAmalgam}
Assume that $\mu$ can be extended to a bounded operator for some $1\leq \theta_1,\ldots, \theta_m$ with
$1/{\theta}=\sum_{j=1}^m1/{\theta_j}<\infty$. Let $1\leq q_j \leq\alpha_j<p_j\leq\infty,\ j = 1,\ldots,m$ with
$1/q=\sum_{j=1}^m 1/{q_j}, 1/p=\sum_{j=1}^m 1/{p_j}, 1/{\alpha}=\sum_{j=1}^m 1/{\alpha_m}$
and  $p/{p_j}=q/{q_j}=\alpha/{\alpha_j}$. If $\vec{\omega}\in A_{\vec{Q}\cap (A_\infty)^m}$, then $\mu$ is a bounded operator from $(L^{q_1}(\omega_1),L^{p_1})^{\alpha_1}\times\cdots\times (L^{q_m}(\omega_m),L^{p_m})^{\alpha_m}$ to $(L^q(\nu_{\vec{\omega}}),L^p)^{\alpha}$ or $(L^{q,\infty}(\nu_{\vec{\omega}}),L^p)^{\alpha}$.
\end{theorem}

\begin{proof}
Let $x\in B=B(y,t)$. Since there exists at lease one $y_j\in B(x,t)\bigcap(2^{k+1}B\backslash 2^kB)\neq\emptyset$ for $k=1,2,\ldots$, we have $t\geq\min_{j=1,\cdot,m} |y_j-x|\geq 2^{k-1}r.$ Applying Minkowshi's inequality and the size condition of $\Omega,$ it leads us to
\begin{align*}
&\mu_K(\vec{f}^D)(x)\\
&\leq\int_{(\mathbb R^n)^m}\bigg(\int_0^\infty\bigg|\frac1{t^m}\frac{\Omega(x-y_1,\cdots,x-y_m)\chi_{B(x,t)^m}}{|(x-y_1,\cdots,x-y_m)|^{m(n-1)}}\bigg|^2\frac{dt}{t}\bigg)^\frac12\prod_{j=1}^m|f^{d_j}_j(y_j)|dy_1\cdots dy_m\\
&\leq C\int_{(\mathbb R^n)^m}\frac1{(2^kr)^{m(n-1)}}\bigg(\int_{2^{k-1}r}^\infty\frac{dt}{t^{2m+1}}\bigg)^\frac12\prod_{j=1}^m|f^{d_j}_j(y_j)|dy_1\cdots dy_m\\
&\leq C\prod_{j\in\sigma^c}\bigg(\frac{1}{|2B|}\int_{2B}|f_j(y_j)|dy_j\bigg)\prod_{j\in\sigma}\sum_{k=1}^\infty\frac{1}{2^{kn|\sigma^c|}}\bigg(\frac{1}{|2^{k+1}B|}\int_{2^{k+1}B\backslash 2^kB}|f_j(y_j)|dy_j\bigg).
\end{align*}
Thus by Theorem \ref{OperatoronAmalgam}, we finish the proof of Theorem \ref{ConvolutionMarcinkiewiczAmalgam}.
\end{proof}

\textbf{II: Nonconvolution Type}
\hskip\parindent

In \cite{XY2015}, Xue and Yan also introduced multilinear Marcinkiewicz function with the kernel of nonconvolution type as a example of integral smooth condition of C-Z type I.
\begin{definition}\cite{XY2015}
Let $K$ be a function defined on $\mathbb{R}^n\times\mathbb{R}^{mn}$ with $\text{supp}K \subset \mathfrak{B}:=\{(x,y_1,\ldots,y_m):\sum_{j=1}^m|x-y_j|^2\leq 1\}$. $K$ is called a multilinear Marcinkiewicz kernel
if for some $\delta\in(0,mn)$ and some positive constants $A$, $\gamma_0$ and $B_1>2$,
\begin{align*}
|K(x,\vec{y})|\leq\frac{A}{\bigg(\sum_{j=1}^m|x-y_j|\bigg)^{mn-\delta}},
\end{align*}
\begin{align*}
|K(z,\vec{y})-K(x,\vec{y})|\leq\frac{A|z-x|^{\gamma_0}}{\bigg(\sum_{j=1}^m|x-y_j|\bigg)^{mn-\delta+\gamma_0}},
\end{align*}
whenever $(x, y_1,\ldots,y_m)\subset \mathfrak{B}$ and $|x-z|\leq\max_{j=1}^m|x-y_j|/{B_1}$;
and
\begin{align*}
|K(z,y_1,\ldots,y_j,\ldots,y_m)-K(x,y_1,\ldots,y_j',\ldots,y_m)|
\leq\frac{A|y_j-y_j'|^{\gamma_0}}{\bigg(\sum_{j=1}^m|x-y_j|\bigg)^{mn-\delta+\gamma_0}},
\end{align*}
whenever $(x, y_1,\ldots,y_m)\subset \mathfrak{B}$ and $|y_j-y_j'|\leq |x-y_j|/{B_1}$, for $j=1,\ldots,m$.
\end{definition}
The multilinear Marcinkiewicz integral of nonconvolution type $\mu_K$ is defined by
\begin{align}\label{NonconvolutionMarcinkiewicz}
\mu_K(\vec{f})(x):=\bigg(\int_0^\infty\bigg|\int_{(\mathbb{R}^n)^m}K_t(x,y_1,\ldots,y_m)\prod_{j=1}^mf_j(y_j)d\vec{y}\bigg|^2\frac{dt}{t}\bigg)^{1/2}.
\end{align}

\begin{lemma}\label{NonconvolutionMarcinkiewiczLebesgue}\cite{XY2015}
Suppose that for some $1\leq s_1,\ldots,s_m<\infty$ with $1/s=\sum_{j=1}^m1/{s}$, $\mu_k$
is bounded from $L^{s_1}(\mathbb{R}^n)\times\cdots\times L^{s_m}(\mathbb{R}^n)$. Let $1\leq q_j <\infty,\ j = 1,\ldots,m$ with
$1/q=\sum_{j=1}^m 1/{q_j}$ and $\vec\omega\in A_{\vec{Q}}$, then $\mu_K$ is a bounded operators from $L^{q_1}(\omega_1)\times\cdots\times L^{q_m}(\omega_m)$ to $L^q(\nu_{\vec{\omega}})$ or $L^{q,\infty}(\nu_{\vec{\omega}})$.
\end{lemma}

This non-convolution type Marcinkiewicz integral $\mu_K$ is also a bounded operator on weighted Amalgam space.

\begin{theorem}\label{NonconvolutionMarcinkiewiczAmalgam}
Suppose that for some $1\leq s_1,\ldots,s_m<\infty$ with $1/s=\sum_{j=1}^m1/{s}$, $\mu_k$
is bounded from $L^{s_1}(\mathbb{R}^n)\times\cdots\times L^{s_m}(\mathbb{R}^n)$. Let $1\leq q_j \leq\alpha_j<p_j\leq\infty,\ j = 1,\ldots,m$ with
$1/q=\sum_{j=1}^m 1/{q_j}, 1/p=\sum_{j=1}^m 1/{p_j}, 1/{\alpha}=\sum_{j=1}^m 1/{\alpha_m}$
and  $p/{p_j}=q/{q_j}=\alpha/{\alpha_j}$, $\vec{\omega}\in A_{\vec{Q}}$.
Then $\mu_K$ is a bounded operator from $(L^{q_1}(\omega_1),L^{p_1})^{\alpha_1}\times\cdots\times (L^{q_m}(\omega_m),L^{p_m})^{\alpha_m}$ to $(L^q(\nu_{\vec{\omega}}),L^p)^{\alpha}$ or $(L^{q,\infty}(\nu_{\vec{\omega}}),L^p)^{\alpha}$.
\end{theorem}

\begin{proof}
By Theorem 1.3 of \cite{XY2015}, we know that $K_t$ satisfies the
integral smooth condition of C-Z type I with $0<\gamma\leq \min\{\delta,\gamma_0\}$, therefore
there exist positive constants $A>1$ such that
\begin{align*}
\bigg(\int_0^\infty|K_t(x,\vec{y})|^2\frac{dt}{t}\bigg)^{\frac{1}{2}}\leq\frac{A}{\bigg(\sum_{j=1}^m|x-y_j|\bigg)^{mn}}.
\end{align*}
Then by a similar proof to Theorem \ref{squarefunctionIAmalgam}, we get
$$
|\mu_K(\vec{f}^D)(x)|
\leq C\leq\prod_{j\in\sigma^c}\frac{1}{|2B|}\int_{2B}|f_j(y_j)|dy_j
\prod_{j\in\sigma}\sum_{k=1}^\infty\frac{1}{2^{kn\sigma^c}}\frac{1}{|2^{k+1}B|}\int_{2^{k+1}B\backslash 2^kB}|f_j(y_j)|dy_j.
$$
Thus by Theorem \ref{OperatoronAmalgam}, we finish the proof of Theorem \ref{NonconvolutionMarcinkiewiczAmalgam}.
\end{proof}

\begin{remark}\label{MultilinearSingular}
The multilinear singular integral operators including the multilinear Calder\'on-Zygmund operators and multilinear
singular integrals with nonsmooth kernels also satisfy the condition (\ref{AmalgamCondition}), so the results in Theorem \ref{OperatoronAmalgam}, \ref{multilinearCommutatoronAmalgam} and \ref{IteratedCommutatoronAmalgam} also hold for these operators and their commutators by some known weighted estimate from \cite{LOPTT2009, PPTT2014,CW2012,GLY2011,LWZ2014,AD2010,DGGLY2009,DGY2010} and so on.
\end{remark}




\section{Multilinear fractional type integral operators on amalgam type space \label{s4}}
\hskip\parindent

In this section, we will consider the fractional type integral operators and their commutators' behaviors on amalgam space and amalgam-Campanato space.
The following multilinear fractional integral has been studied by Kenig and Stein in \cite{KS1999,M2009}, in which the weak and strong estimates were given for a class of generalized fractional integral operators.

\begin{definition}\label{Fractional}
Let $0<\gamma<mn$, the multilinear fractional integral operator is defined by
$$
I_\gamma(\vec{f})(x):=\int_{(\mathbb{R}^n)^m}\frac{1}{|(x-y_1,\ldots,x-y_m)|^{mn-\gamma}}\prod_{j=1}^mf_j(y_j)d\vec{y},
$$
where $d\vec{y}=dy_1\cdots dy_m$ and $|(y_1,\ldots,y_m)|=\sum_{j=1}^m|y_j|$.
\end{definition}

The boundedness of multilinear commutator and iterated commutator of multilinear fractional integral above on weighted Lebesgue space were consider in \cite{M2009,CX2010,CW2013}, respectively. Chen and Xue \cite{CX2010} considered the multilinear fractional integral with homogeneous kernels as follows and obtained its boundedness on weighted Lebesgue space.

\begin{definition}\label{Fractional}
Assume each $\Omega_j(x)\in L^{\theta'}(S^{n-1})~(j=1,\ldots,m)$ for some $\theta'>1$ is
a homogeneous function with degree zero on $\mathbb{R}^n$. Then for any $x\in\mathbb{R}^n$
and $0<\alpha<mn$, define the multilinear fractional integral with
homogeneous kernels as
$$
I_{\Omega,\gamma}(\vec{f})(x):=\int_{(\mathbb{R}^n)^m}\frac{\prod_{j=1}^m\Omega_j(x-y_j)f_j(y_j)}{|(x-y_1,\ldots,x-y_m)|^{mn-\gamma}}d\vec{y},
$$
where $|(y_1,\ldots,y_m)|=\sum_{j=1}^m|y_j|$.
\end{definition}

The following classes of weights $A_{p,q}$ and $A_{\vec{P},q}$ are considered in this section.
\begin{definition}\label{Fractionalweight}\cite{MW1974}
Let $1\leq p<\infty$ and $q>0$. Suppose that $\omega$ is a nonnegative function on $\mathbb{R}^n$. We say that $\omega\in A_{p,q}$ if it satisfies
$$
\sup_B\bigg(\frac{1}{|B|}\int_B\omega^q(x)dx\bigg)^{1/q}\bigg(\frac{1}{|B|}\int_B\omega^{-p'}(x)dx\bigg)^{1/{p'}}<\infty.
$$
If $p=1$, $\bigg(\frac{1}{|B|}\int_B\omega^{-p'}(x)dx\bigg)^{1/{p'}}$ is understood as $(\inf_B\omega(x))^{-1}$.
\end{definition}
\begin{definition}\label{FractionalMultiplierweight}\cite{CX2010}
Let $1\leq p_1,\ldots,p_m<\infty$ with $1/p=\sum_{j=1}^m1/{p_j}$ , and $q>0$. Suppose that $\vec{\omega}=(\omega_1,\ldots,\omega_m)$ and
each $¦Ø_j~(i = 1,\ldots,m)$ is a nonnegative function on $\mathbb{R}^n$. We say that $\vec{\omega}\in A_{\vec{p},q}$ if it satisfies
$$
\sup_B\bigg(\frac{1}{|B|}\int_B\nu_{\vec{\omega}}^q(x)dx\bigg)^{1/q}\prod_{j=1}^m\bigg(\frac{1}{|B|}\int_B\omega_j^{-p_j'}(y_j)dy_j\bigg)^{1/{p_j'}}<\infty
$$
where $\nu_{\vec{\omega}}=\prod_{j=1}^m\omega_j$. If $p_j=1$, $\bigg(\frac{1}{|B|}\int_B\omega_j^{-p_j'}\bigg)^{1/{p_j'}}$ is understood as $(\inf_B\omega_j)^{-1}$.
\end{definition}

The following properties of the weight class $A_{\vec{P},q}$ will be used later.
\begin{lemma}\label{FractionalMultiplierweightandWeight}\cite{WY2013}
Let $m\geq 2$, $q_1,\ldots,q_m\in[1,\infty)$, $q\in(0,\infty)$ with $1/q=\sum_{j=1}^m1/{q_j}$, $\omega_1^{q_1},\ldots,\omega_m^{q_m}\in A_\infty$ and $\nu_{\vec{\omega}}=\prod_{j=1}^m\omega_j$, then
$$
\prod_{j=1}^m\bigg(\int_B\omega_j^{q_j}(x)dx\bigg)^{1/{q_j}}\leq C\bigg(\int_B\nu_{\vec{\omega}}^q(x)dx\bigg)^{1/q}.
$$
\end{lemma}

\begin{lemma}\label{Fractionalweights}
Let $0 <\gamma< mn,$ $p_1,\ldots,p_m\in[1,\infty),$ $1/p=\sum_{j=1}^m1/p_j$ and $1/q = 1/p-\gamma/n.$ Then $\vec\omega=(\omega_1,\ldots,\omega_m)\in A_{\vec P,q}$ if and only if
$$
\begin{cases}
\nu_{\vec\omega}^q\in A_{mq},\\
\omega_j^{-p_j'}\in A_{mp_j'},\ j=1,\ldots,m,
\end{cases}
$$
where $\nu_{\vec\omega}=\prod_{j=1}^m\omega_j.$
\end{lemma}

Our main theorem for the multilinear fractional integrals is stated as follows.
\begin{theorem}\label{FractionalAmalgam}
Let $1\leq\theta< p_j\leq\alpha_j<s_j\leq\infty,$ $1/q_j=1/p_j-\gamma/(mn),$ $1/q=\sum_{j=1}^m1/{q_j}$, $1/s=\sum_{j=1}^m1/{s_j}$ and $1/{\alpha^*}=1/\alpha-\gamma/n:=\sum_{j=1}^m(1/{\alpha_j}-\gamma/{mn})$ with $\alpha/{\alpha_j}=p/{p_j}=s/s_j$. If $\vec{\omega}^{\theta}\in A_{\frac{\vec{P}}{\theta},\frac{q}{\theta}}$ with $\omega_1^{p_1},\ldots,\omega_m^{p_m}\in A_\infty,$ and $\mathcal{T}_{\gamma},$ which is bounded from $L^{p_1}(\omega_m^{p_1})\times\cdots L^{p_m}(\omega_m^{p_m})$ to $L^{q}(\nu_{\vec{\omega}}^{q})$, satisfies condition
\begin{align}\label{FractionalAmalgamcondition}
|\mathcal{T}_\gamma(\vec{f}^D)(x)|&\leq C\prod_{j\in\sigma^c}\bigg(\frac{1}{|2B|^{1-\gamma\theta/{mn}}}\int_{2B}|f_j(y_j)|^{\theta}dy_j\bigg)^{1/{\theta}}\\
&\hspace*{12pt}\times\prod_{j\in\sigma}\sum_{k=1}^\infty\frac{k}{2^{kn|\varrho|}}\bigg(\frac{1}{|2^{k+1}B|^{1-\gamma\theta/{mn}}}\int_{2^{k+1}B\backslash 2^kB}|f_j(y_j)|^{\theta} dy_j\bigg)^{1/{\theta}},  \nonumber
\end{align}
where $|\varrho|=|\sigma^c|(1-\gamma/{mn}).$ Then $\mathcal T_{\gamma}$ maps $(L^{p_1}(\omega_1^{p_1},\omega_1^{q_1}),L^{s_1})^{\alpha_1}\times\cdots\times(L^{p_m}(\omega_m^{p_m},\omega_m^{q_m}),L^{s_m})^{\alpha_m}$ to $(L^q(\nu_{\vec{\omega}}^q),L^s)^{\alpha^*}$.
\end{theorem}

\begin{proof}
Let $x\in B(y,r)$, in view of \eqref{FractionalAmalgamcondition}, we have
\begin{align*}
|\mathcal{T}_\gamma(\vec{f})(x)|
&\leq |\mathcal{T}_\gamma(\vec{f}^0)(x)|+\sum_{(d_1,\ldots,d_m)\in\mathcal{I}}|\mathcal{T}_\gamma(f_1^{d_1},\ldots,f_m^{d_m})(x)|\\
&\leq |\mathcal{T}_\gamma(\vec{f}^0)(x)|\\
&\hspace*{12pt}+\sum_{(d_1,\ldots,d_m)\in\mathcal{I}}\prod_{j\in\sigma^c}\bigg(\frac{1}{|2B|^{1-\gamma/{mn}}}\int_{2B}|f_j(y_j)| dy_j\bigg)\\
&\hspace*{12pt}\times\prod_{j\in\sigma}\sum_{k=1}^\infty\frac{1}{2^{kn\varrho}}\bigg(\frac{1}{|2^{k+1}B|^{1-\gamma/{mn}}}\int_{2^{k+1}B\backslash 2^kB}|f_j(y_j)| dy_j\bigg)\\
&:=F_1(x)+F_2(x).
\end{align*}
By Lemma \ref{FractionalMultiplierweightandWeight} and the condition $\alpha/{\alpha_j}=p/p_j=s/{s_j}$, we have an important inequality as follows
\begin{align*}
\nu_{\vec{\omega}}^q(B)^{\frac{1}{\alpha^*}-\frac{1}{q}-\frac{1}{s}}=\nu_{\vec{\omega}}^q(B)^{\frac{1}{\alpha}-\frac{1}{p}-\frac{1}{s}}
\leq C\prod_{j=1}^m\omega_j^{q_j}(B)^{\frac{1}{\alpha_j}-\frac{1}{p_j}-\frac{1}{s_j}}.
\end{align*}

Since $T_\gamma$ is a bounded operators from $L^{p_1}(\omega_1^{p_1})\times\cdots L^{p_m}(\omega_m^{p_m})$ to $L^q(\nu_{\vec{\omega}}^q)$, we
have
\begin{align*}
\big(\nu^q_{\vec{\omega}}(B)\big)^{\frac{1}{\alpha*}-\frac{1}{q}-\frac{1}{s}}\|F_1\chi_B\|_{L^q(\nu_{\vec{\omega}}^q)}&\leq C\big(\nu^q_{\vec{\omega}}(B)\big)^{\frac{1}{\alpha*}-\frac{1}{q}-\frac{1}{s}}\prod_{j=1}^m\|f_j\chi_{2B}\|_{L^{p_j}(\omega_j^{p_j})}\\
&\leq C\prod_{j=1}^m(\omega_j^{q_j}(2B))^{\frac{1}{\alpha_j}-\frac{1}{p_j}-\frac{1}{s_j}}\|f_j\chi_{2B}\|_{L^{p_j}(\omega_j^{p_j})},
\end{align*}
therefore
$$
\|F_1\|_{(L^q(\nu_{\vec{\omega}}),L^s)^{\alpha*}}\leq C\prod_{j=1}^m\|f_j\|_{(L^{p_j}(\omega_j^{p_j},\omega_j^{q_j}),L^{s})^{\alpha_j}}.
$$

To estimate the second part $F_2(x)$, an applying of H\"older's inequality and $A_{\vec P,q}$ condition tell us that
\begin{align*}
F_2(x)&\leq C\sum_{(d_1,\ldots,d_m)\in\mathcal I}\prod_{j\in\sigma^c}\|f_j\chi_{2B}\|_{L^{p_j}(\omega_j^{p_j})}|2B|^{-1+\frac\gamma{mn}}\bigg(\int_{2B}\omega_j^{-p_j'}(y_j)dy_j\bigg)^\frac1{p_j'}\\
&\hspace*{12pt}\times\prod_{j\in\sigma}\sum_{k=1}^\infty\frac{k}{2^{kn|\varrho|}}
\|f_j\chi_{2^{k+1}B}\|_{L^{p_j}(\omega_j^{p_j})}|2^kB|^{-1+\frac\gamma{mn}}\bigg(\int_{2^kB}\omega_j^{-p_j'}(y_j)dy_j\bigg)^\frac1{p_j'}\\
&\leq C\sum_{k=1}^\infty k(\nu^q_{\vec\omega}(2^{k+1}B))^{-\frac1q}\prod_{j=1}^m\|f_j\chi_{2^{k+1}B}\|_{L^{p_j}(\omega_j^{p_j})}.
\end{align*}
Taking $L^q(\nu_{\vec{\omega}})$-norm on the ball $B$ of both side inequality and then multiplying both side by $\big(\nu^q_{\vec{\omega}}(B)\big)^{\frac{1}{\alpha}-\frac{1}{q}-\frac{1}{s}}$, we obtain
\begin{align*}
\big(\nu^q_{\vec{\omega}}(B)\big)^{\frac{1}{\alpha}-\frac{1}{q}-\frac{1}{s}}&\|F_2\chi_B\|_{L^q(\nu_{\vec{\omega}})}\\
&\leq C\sum_{k=1}^\infty k\bigg(\frac{\nu^q_{\vec{\omega}}(B)}{\nu^q_{\vec{\omega}}(2^{k+1}B)}\bigg)^{\frac{1}{\alpha}-\frac{1}{s}}\prod_{j=1}^m\big(\omega_j^{p_j}(2^{k+1}B)\big)^{\frac1\alpha-\frac1{q_j}-\frac1{s_j}}\|f_j\chi_{2^{k+1}B}\|_{L^{p_j}(\omega_j^{p_j})}\\
&\leq C\prod_{j=1}^m\|f_j\|_{(L^{p_j}(\omega_j^{p_j},\omega_j^{q_j}),L^{s})^{\alpha_j}},
\end{align*}
where we use Lemma \ref{FractionalMultiplierweightandWeight} and \ref{Fractionalweights}. Therefore, we complete the proof of this theorem.
\end{proof}
\begin{remark}
We should point out that it is clearly that the fractional type integral operators $I_{\gamma}$ and $I_{\Omega, \gamma}$ above satisfy our condition \eqref{FractionalAmalgamcondition}.
\end{remark}

We will give our results for the multilinear commutators and iterated commutators generated by the fractional multilinear operators and function vector $\vec b\in BMO^m$ without the proof. In fact, the process is very similar to Theorem \ref{multilinearCommutatoronAmalgam} and \ref{IteratedCommutatoronAmalgam} using some properties of $A_{\vec P,q}.$

\begin{theorem}\label{MultilinearCommutatorFractionalAmalgam}
Let $1\leq\theta< p_j\leq\alpha_j<s_j\leq\infty,$ $1/q_j=1/p_j-\gamma/(mn),$ $1/q=\sum_{j=1}^m1/{q_j}$, $1/s=\sum_{j=1}^m1/{s_j}$ and $1/{\alpha^*}=1/\alpha-\gamma/n:=\sum_{j=1}^m(1/{\alpha_j}-\gamma/{mn})$ with $\alpha/{\alpha_j}=p/{p_j}=s/s_j$. If $\vec{\omega}^{\theta}\in A_{\frac{\vec{P}}{\theta},\frac{q}{\theta}}$ with $\omega_1^{p_1},\ldots,\omega_m^{p_m}\in A_\infty$, and $\mathcal{T}_{\gamma}$ satisfying \eqref{FractionalAmalgamcondition} is bounded from $L^{p_1}(\omega_m^{p_1})\times\cdots L^{p_m}(\omega_m^{p_m})$ to $L^{q}(\nu_{\vec{\omega}}^{q})$,
then multilinear commutator $\mathcal T_{\gamma,\Sigma\vec{b}}$ and $\mathcal T_{\gamma,\Pi\vec{b}}$  is bounded from $(L^{p_1},L^{s_1})^{\alpha_1}\times\cdots\times(L^{p_m},L^{s_m})^{\alpha_m}$ to $(L^q,L^s)^{\alpha^*}$.
\end{theorem}

Next we will give an endpoint estimate for the multilinear fractional type operator without weights. To do this, we introduce the amalgam-Campanato space at first.
\begin{definition}\label{AmalgamCamponato}
Let $n<\beta<\infty$ and $0< q\leq \alpha\leq p\leq\infty$. We define the Amalgam-Campanato space $(C^q,L^p)^{\alpha,\beta}:=(C^q,L^p)^{\alpha,\beta}(\mathbb R^n)$ as the space of all measurable functions $f$ satisfying $\|f\|_{(C^q,L^p)^{\alpha,\beta}}<\infty$, where
$$
\|f\|_{(C^q,L^p)^{\alpha,\beta}}:=\underset{r>0}{\sup}\leftidx{_r}\|f\|_{(C^q,L^p)^{\alpha,\beta}}
$$
with
\begin{align*}
\leftidx{_r}\|f\|_{(C^q,L^p)^{\alpha,\beta}}:=
\begin{cases}\displaystyle
\bigg(\int_{\mathbb{R}^n}|B(y,r)|^{\frac{1}{\alpha}-\frac{1}{\beta}-\frac{1}{q}-\frac{1}{p}}\|(f-f_{B(y,r)})\chi_{B(y,r)}\|_{L^q}\big)^pdy\bigg)^{\frac{1}{p}},\quad &p<\infty,\\
\mathop{\rm{ess}\sup}\limits_{y\in\mathbb{R}^n}|B(y,r)|^{\frac{1}{\alpha}-\frac{1}{\beta}-\frac{1}{q}}\|(f-f_{B(y,r)})\chi_{B(y,r)}\|_{L^q}, \quad &p=\infty.
\end{cases}
\end{align*}
\end{definition}

\begin{remark}
It is easy to see that the space goes back to the classical Campanato space $C^{q,\frac{1}{\beta}-\frac{1}{\alpha}}$ when $p=\infty.$
\end{remark}

Some results established in \cite{T2008} can also be viewed as consequences of the following theorem in the case $p=\infty.$
\begin{theorem}\label{FractionalAmalgam-Campanato}
Let $m\in\mathbb N_+,$ $0<\gamma<mn,$ $\beta_j>mn,$ $1\leq\alpha_j<\beta_j,$ $1/(\alpha_j,\beta_j)=1/\alpha_j-1/\beta_j$, $1\leq q_j\leq\alpha_j<p_j\leq\infty,$ $1/\beta=1/\beta_1+\cdots+1/\beta_m,$ $1/{\alpha^*}=1/{\alpha}-\gamma/n=\sum_{j=1}^m(1/{\alpha_j}-\gamma/{mn})$, $1/q=\sum_{j=1}^m1/{q_j}$ and $1/p=\sum_{j=1}^m1/{p_j}$. If $(\alpha_j,\beta_j)\vee\alpha_j*<p_j,j=1,2,\cdots,m,$ then $I_{\gamma}$ is bounded from $(L^{q_1},L^{p_1})^{(\alpha_1,\beta_1)}\times\cdots\times(L^{q_m},L^{p_m})^{(\alpha_m,\beta_m)}$ to $(C^q,L^p)^{\alpha*,\beta}$ for $q=1$.
\end{theorem}

\begin{proof}
Fix $y\in\mathbb R^n$ and $r>0.$ Sex $B=B(y,r)$ and let $x\in B(y,r)$, we decompose $f_j=f_j^0+f_j^\infty:=f_j\chi_{2B(y,r)}+f_j\chi_{2B(y,r)^c}$ for $j=1,\ldots,m$.
We write
\begin{align*}
\frac{1}{|B|^{1/q+1/p+1/\beta-1/\alpha*}}&\int_B\big|I_\gamma(\vec{f})(x)-\big(I_\gamma(\vec{f})\big)_B\big|dx\\
&\leq\frac{1}{|B|^{1/q+1/p+1/\beta-1/\alpha*}}\int_B\big|I_\gamma(\vec{f}^0)(x)-\big(I_\gamma(\vec{f}^0)\big)_B\big|dx\\
&\hspace*{12pt}+\sum_{(d_1,\ldots,d_m)\in\mathcal{I}}
\frac{1}{|B|^{1/q+1/p+1/\beta-1/\alpha*}}\int_B\big|I_\gamma(\vec{f}^D)(x)-\big(I_\gamma(\vec{f}^D)\big)_B\big|dx\\
&:=J_1(y,r)+J_2(y,r).
\end{align*}
First, we deal with $J_1(y,r)$. Let $\gamma=\sum_{j=1}^m1/{\gamma_j}$ with $n/{q_j}<\gamma_j<n$ for $j=1,\ldots,m$.
By the definition of $I_\gamma$, we get
\begin{align*}
J_1(y,r)&\leq\frac{C}{|B|^{1/q+1/p+1/\beta-1/\alpha*}}\int_B|I_\gamma(\vec{f}^0)(x)|dx\\
&\leq\frac{C}{|B|^{1/q+1/p+1/\beta-1/\alpha*}}\int_B\int_{(2B)^m}\frac{\prod_{j=1}^m|f_j(z_j)|}{|(x-z_1,\ldots,x-z_m)|^{mn-\gamma}}d\vec{z}dx\\
&\leq\frac{C}{|B|^{1/q+1/p+1/\beta-1/\alpha*}}\int_B\Bigg(\prod_{j=1}^m\bigg(\int_{2B}|f_j(z_j)|^{q_j}dz_j\bigg)^{1/{q_j}}
\bigg(\int_{B}\frac{1}{|x-z_j|^{(n-\gamma_j)q_j'}}dz_j\bigg)^{1/{q_j'}}\Bigg)dx.
\end{align*}
Since $n/{q_j}<\gamma_j<n,$ we have
$$
\bigg(\int_{B}\frac{1}{|x-z_j|^{(n-\gamma_j)q_j'}}dz_j\bigg)^{1/{q_j'}}\leq C|2B|^{\frac{\gamma_j}{n}-\frac{1}{q_j}}.
$$
Then the inequality above leads us to
\begin{align*}
J_1(y,r)
&\leq\frac{C}{|B|^{1/p+1/\beta-1/\alpha*}}\prod_{j=1}^m\|f_j\chi_{2B}\|_{L^{q_j}}|2B|^{\frac{\gamma_j}{n}-\frac{1}{q_j}}\\
&\leq C\prod_{j=1}^m\frac{1}{|2B|^{1/{q_j}+1/{p_j}+1/{\beta_j}-1/{\alpha_j}}}\|f_j\chi_{2B}\|_{L^{q_j}}.
\end{align*}
Taking $L^p(\mathbb{R}^n)$-norm on both side inequality above, we obtain
\begin{align*}
\bigg(\int_{\mathbb{R}^n}(J_1(y,r))^qdy\bigg)^{1/q}
&\leq C\prod_{j=1}^m
\Bigg(\int_{\mathbb{R}^n}\bigg(\frac{1}{|2B|^{1/{q_j}+1/{p_j}+1/{\beta_j}-1/{\alpha_j}}}\|f_j\chi_{2B}\|_{L^{q_j}}\bigg)^{p_j}dy\Bigg)^{1/p_j}\\
&\leq C\prod_{j=1}^m\|f_j\|_{(L^{q_j},L^{p_j})^{\alpha_j,\beta_j}}.
\end{align*}
Next we estimate the term $J_2(y,r)$. We set
$$
F:=\big|I_\gamma(\vec{f}^D)(x)-I_\gamma(\vec{f}^D)(x')\big|,
$$
where $(d_1,\ldots,d_m)\in\mathcal{I}$. We assume that  $d_1=\cdots=d_l=\infty, d_{l+1}=\cdots=d_m=0$ with $1\leq l\leq m.$ Since $x,x'\in B$ and some $z_j\in 2^kB\backslash 2^{k-1}B$ for $k\geq 1$, we have $|x-x'|\leq 2r$ and $|(x-z_1,\cdots,x-z_m)|\geq 2^{k-1}r.$ then
\begin{align*}
F&\leq\int_{\mathbb{R}^n\backslash 2B}\frac{|x-x'|}{|(x-z_1,\ldots,x-z_m)|^{mn+1-\gamma}}\prod_{j=1}^m|f_j(z_j)|dz_j\\
&\leq C|2B|^{1/n}\prod_{j=l+1}^m\int_{2B}|f_j(z_j)|dz_j
\prod_{j=1}^l\sum_{k=1}^\infty\frac{1}{|2^kB|^{(mn+1)/{ln}-\gamma/ml}}\int_{2^{k+1}B\backslash 2^kB}|f_j(z_j)|dz_j\\
&\leq C|2B|^{\frac{1}{n}}\prod_{j=1}^m\sum_{k=1}^\infty\frac{1}{|2^kB|^{(mn+1)/{mn}-\gamma/mn}}\|f_j\chi_{2^{k+1}B}\|_{L^{p_j}}|2^{k+1}B|^{1-\frac{1}{p_j}}\\
&\leq C\prod_{j=1}^m\sum_{k=1}^\infty\frac{|2B|^{1/nm}}{|2^kB|^{1/{mn}-\gamma/mn+1/\alpha_j-1/p_j-1/\beta_j}}
\bigg(\frac{1}{|2^{k+1}B|^{1/{q_j}+1/{p_j}+1/{\beta_j}-1/{\alpha_j}}}\|f_j\chi_{2^{k+1}B}\|_{L^{q_j}}\bigg)\\
&\leq C\frac{1}{|2B|^{1/{\alpha*}-1/p-1/\beta}}\prod_{j=1}^m\sum_{k=1}^\infty\frac{1}{2^{kn\rho_j}}
\bigg(\frac{1}{|2^{k+1}B|^{1/{q_j}+1/{p_j}+1/{\beta_j}-1/{\alpha_j}}}\|f_j\chi_{2^{k+1}B}\|_{L^{q_j}}\bigg),
\end{align*}
where $\rho_j={1/{mn}+1/{\alpha_j*}-1/p_j-1/\beta_j}.$
Therefore
$$
J_2(y,r)\leq C
\prod_{j=1}^m\sum_{k=1}^\infty\frac{1}{2^{kn\rho_j}}
\bigg(\frac{1}{|2^{k+1}B|^{1/{q_j}+1/{p_j}+1/{\beta_j}-1/{\alpha_j}}}\|f_j\chi_{2^{k+1}B}\|_{L^{q_j}}\bigg).
$$
Since $\alpha_j*<p_j$ and $\beta_j>mn$ imply $\rho_j>0,$ by H\"older's inequality we obtain
$$
\bigg(\int_{\mathbb{R}^n}(J_2(y,r))^pdy\bigg)^{1/p}\leq C\prod_{j=1}^m\|f_j\|_{(L^{p_j},L^{q_j})^{\alpha_j,\beta_j}}.
$$
Taking $L^p(\mathbb{R}^n)$-norm on both side inequality above, we obtain
\begin{align*}
\bigg(\int_{\mathbb{R}^n}(J_2(y,r))^pdy\bigg)^{1/p}\leq C\prod_{j=1}^m\|f_j\|_{(L^{q_j},L^{p_j})^{\alpha_j,\beta_j}},
\end{align*}
and complete our proof of Theorem \ref{FractionalAmalgam-Campanato}.
\end{proof}

\begin{remark}\label{One}
We should point out that even if $m=1$, our results Theorem \ref{FractionalAmalgam}, Theorem \ref{MultilinearCommutatorFractionalAmalgam} and Theorem \ref{FractionalAmalgam-Campanato} are also new.
\end{remark}


\textbf{Acknowledgement}: The authors would like to thank the referee for his/her helpful suggestions.


\bigskip

\noindent Songbai Wang

\medskip

\noindent College of Mathematics and Statistics, Hubei Normal University, Huangshi, 435002, People¡¯s Republic of China

\smallskip

\noindent{\it E-mail addresses}: \texttt{haiyansongbai@163.com} (S.-B. Wang)

\bigskip

\noindent Peng Li

\medskip

\noindent Graduate School, China Academy of Engineering Physics, Beijing,
100088, People's Republic of China\\
Institute of Applied Physics and Computational Mathematics, Beijing,
100088, People's Republic of China\\

\smallskip

\noindent{\it E-mail addresses}: \texttt{lipengmath@126.com} (P. Li)
\end{document}